\documentclass{amsart}

\usepackage{amssymb}
\usepackage{amsmath}
\usepackage{amsthm}
\usepackage{amsfonts}

\usepackage{a4wide}
\usepackage{enumerate}

\newtheorem{theorem}{Theorem}[section]
\newtheorem{proposition}{Proposition}[section]
\newtheorem{corollary}{Corollary}[section]

\theoremstyle{definition}
\newtheorem{definition}{Definition}[section]
\newtheorem{remark}{Remark}[section]
\newtheorem{example}{Example}[section]

\numberwithin{equation}{section}

\begin{document}


\title{New number fields with known \(p\)-class tower}

\author{Daniel C. Mayer}
\address{Naglergasse 53\\8010 Graz\\Austria}
\email{algebraic.number.theory@algebra.at}
\urladdr{http://www.algebra.at}

\thanks{Research supported by the Austrian Science Fund (FWF): P 26008-N25}

\subjclass[2000]{Primary 11R37, 11R29, 11R11, 11R16, 11R20;
secondary 20D15, 20F05, 20F12, 20F14, 20-04}
\keywords{\(p\)-class field towers, \(p\)-capitulation, \(p\)-class groups, 
quadratic fields, cubic fields, dihedral fields;
finite \(p\)-groups, descendant trees, \(p\)-group generation algorithm}

\date{October 02, 2015}

\begin{abstract}
The \(p\)-class tower \(\mathrm{F}_p^\infty(k)\) of a number field \(k\)
is its maximal unramified pro-\(p\) extension.
It is considered to be known when the \(p\)-tower group, that is
the Galois group \(G:=\mathrm{Gal}(\mathrm{F}_p^\infty(k)\vert k)\),
can be identified by an explicit presentation.
The main intention of this article is to characterize
assigned finite \(3\)-groups uniquely
by abelian quotient invariants of subgroups of finite index,
and to provide evidence of actual realizations
of these groups by \(3\)-tower groups \(G\)
of real quadratic fields \(K=\mathbb{Q}(\sqrt{d})\)
with \(3\)-capitulation type \((0122)\) or \((2034)\).
\end{abstract}

\maketitle



\section{Introduction}
\label{s:Intro}

Given a prime \(p\),
the Hilbert \(p\)-class field tower \(\mathrm{F}_p^\infty(k)\)
is the maximal unramified pro-\(p\) extension
of an algebraic number field \(k\).
It is considered to be \textit{known}
if a presentation of its Galois group is given
(and not merely its length \(\ell_p(k)\)).
The key for determining the Galois group
\[G:=\mathrm{G}_p^\infty(k)=\mathrm{Gal}(\mathrm{F}_p^\infty(k)\vert k),\]
which is briefly called the \(p\)-\textit{class tower group} of \(k\),
is the structure of \(p\)-class groups \(\mathrm{Cl}_p(L)\)
of unramified (abelian or non-abelian) \(p\)-extensions \(L\vert k\),
collected in IPADs (\textit{index-\(p\) abelianization data})
\cite{Ma6}.
Our main goal is to present \textit{new criteria}
for the occurrence of assigned \(3\)-class tower groups \(G\)
in terms of their IPADs
and proofs of their \textit{actual realization} by suitable base fields \(k\).

\textit{Complex} quadratic fields \(k=\mathbb{Q}(\sqrt{d})\)
with negative fundamental discriminant \(d<0\)
were the first objects
whose \(p\)-class tower has been studied for an odd prime \(p\)
by Scholz and Taussky in \(1934\)
\cite{SoTa}.
These authors used group theory to classify the
\(3\)-capitulation type \(\varkappa_1(k)=(\ker(j_i))_{1\le i\le 4}\)
of complex quadratic fields \(k\)
with \(3\)-class group \(\mathrm{Cl}_3(k)\simeq (3,3)\)
and four class extension homomorphisms \(j_i:\mathrm{Cl}_3(k)\to\mathrm{Cl}_3(L_i)\)
to the unramified cyclic cubic superfields \(L_i\), \(1\le i\le 4\),
into \(13\) possible cases,
denoting sections of similar types with upper case letters D,E,F,G,H
and proving that other sections A,B,C are impossible
(B,C in general, A for quadratic fields).


In the present article, we use IPADs of 2\({}^{\mathrm{nd}}\) order to show that
\textit{real} quadratic fields \(K=\mathbb{Q}(\sqrt{d})\)
with one of the \(3\)-capitulation types
c.18, \(\varkappa=(0122)\), and c.21, \(\varkappa=(2034)\),
where \(0\) denotes a total capitulation
\cite[p.477]{Ma2},
have \(3\)-class field towers of \textit{exact length} \(\ell_3(K)=\mathbf{3}\).
These types are strange because they are unique
with the following properties:
Every finite metabelian \(3\)-group \(\mathfrak{G}\)
with one of these types \(\varkappa_1(\mathfrak{G})\)
has coclass \(\mathrm{cc}(\mathfrak{G})=2\) and is \textit{infinitely capable}
with nuclear rank \(\mathrm{NR}(\mathfrak{G})\ge 1\)
\cite{Ma5}.
So \(\mathfrak{G}\) cannot be leaf of a tree.
The \textit{second \(3\)-class group}
\(\mathfrak{G}:=\mathrm{G}_3^2(K)=\mathrm{Gal}(\mathrm{F}_3^2(K)\vert K)\)
\cite{Ma1}
of the real quadratic fields \(K=\mathbb{Q}(\sqrt{d})\), \(d>0\), under investigation
is such a group.

Section c (containing types 18 and 21) was found group-theoretically by Nebelung
\cite{Ne}
in \(1989\), \(55\) years after the Scholz--Taussky sections D,...,H.
However, 
for nearly \(20\) years, examples of fields with these types,
which certainly cannot occur for complex quadratic fields, were unknown,
and we were the first who discovered suitable real quadratic base fields
\(K=\mathbb{Q}(\sqrt{d})\), namely

\(d = \mathbf{540\,365}\) with type c.21 on January \(01\), \(2008\), and

\(d = \mathbf{534\,824}\) with type c.18 on August \(20\), \(2009\)
\cite[Tbl.5, p.499]{Ma1}.

\noindent
Inspired by the
the International Conference on Groups and Algebras in Shanghai
at the end of July \(2015\),
we suddenly had the rewarding idea and the courage
to study their \(3\)-class tower.



\section{Main result}
\label{s:MainResult}

\begin{theorem}
\label{thm:Main}
The following real quadratic fields \(K=\mathbb{Q}(\sqrt{d})\),
with fundamental discriminant \(d>0\),
\(3\)-class group \(\mathrm{Cl}_3(K)\) of type \((3,3)\),
and \(3\)-capitulation type \(\varkappa_1(K)\) belonging to section c,
have a \(3\)-class field tower
\(K<\mathrm{F}_3^1(K)<\mathrm{F}_3^2(K)<\mathrm{F}_3^3(K)=\mathrm{F}_3^\infty(K)\)
of exact length \(\ell_3(K)=3\):

\begin{enumerate}

\item
all the \(4461\) fields of type \(\mathrm{c}.18\), \(\varkappa_1(K)=(0122)\), with \(d<10^9\), and

\item
all the \(4528\) fields of type \(\mathrm{c}.21\), \(\varkappa_1(K)=(2034)\), with \(d<10^9\).

\end{enumerate}

\end{theorem}



\section{Second \(3\)-class groups as tree vertices}
\label{s:MetabelianTrees}

Before we start with required theoretical foundations in section \S\
\ref{s:ArtinPattern},
let us have a glance at the two trees
where all second \(3\)-class groups \(\mathfrak{G}\)
with type \(\varkappa_1(\mathfrak{G})\) in Nebelung's section c
are located as \textit{mainline vertices}.
These trees first arose
in our investigation of metabelian \(3\)-groups \(\mathfrak{G}\)
with type \(\varkappa_1(\mathfrak{G})\) in the Scholz--Taussky section E
\cite[Fig.3.6--3.7, pp.442--443]{Ma4},
which are represented by \textit{terminal leaves}, without exceptions.
They set in with
\(\langle 2187,288\ldots 290\rangle\) on the tree \(\mathcal{T}^2(\langle 243,6\rangle)\)
and with
\(\langle 2187,302\vert 304\vert 306\rangle\) on the tree \(\mathcal{T}^2(\langle 243,8\rangle)\).
The \(3\)-tower length \(\ell_3(k)=3\) of \textit{complex} quadratic fields \(k\)
with second \(3\)-class groups
\(\mathfrak{G}=\mathrm{Gal}(\mathrm{F}_3^2(k)\vert k)\)
of this kind was determined successfully by M.R. Bush and ourselves
\cite[Cor.4.1.1, p.775]{BuMa}
for the first time. We discovered that the corresponding \(3\)-class tower groups
\(G=\mathrm{Gal}(\mathrm{F}_3^\infty(k)\vert k)\)
with \(G/G^{\prime\prime}\simeq\mathfrak{G}\)
are sporadic isolated vertices (outside of coclass trees)
with unbounded higher coclass \(\mathrm{cc}(G)\ge 3\)
but with fixed derived length \(\mathrm{dl}(G)=3\), as visualized in
\cite[Fig.8--9, pp.188--189]{Ma6}.
These results concerning section E were generalized in
\cite[Fig.3--4, pp.754--755]{Ma7}.

Generally, the vertices of the coclass trees in the Figures
\ref{fig:TreeOverviewQ}
and
\ref{fig:TreeOverviewU}
represent isomorphism classes of finite \(3\)-groups.
Two vertices are connected by an edge \(H\to G\) if
\(G\) is isomorphic to the last lower central quotient \(H/\gamma_c(H)\)
where \(c=\mathrm{cl}(H)\) denotes the nilpotency class of \(H\),
and \(\lvert H\rvert=3\lvert G\rvert\), that is,
\(\gamma_c(H)\) is cyclic of order \(3\).
(See also
\cite[\S\ 2.2, p.410--411]{Ma4}
and
\cite[\S\ 4, p.163--164]{Ma5}.)

The vertices of the coclass trees in both figures
are classified by using different symbols:

\begin{enumerate}
\item
big full discs {\Large \(\bullet\)} represent metabelian groups \(\mathfrak{G}\)
with bicyclic centre of type \((3,3)\) and defect \(k(\mathfrak{G})=0\)
\cite[\S\ 3.3.2, p.429]{Ma4},
\item
small full discs {\footnotesize \(\bullet\)} represent metabelian groups \(\mathfrak{G}\)
with cyclic centre of order \(3\) and defect \(k(\mathfrak{G})=1\),
\item
small contour squares {\scriptsize \(\square\)} represent non-metabelian groups \(\mathfrak{H}\).
\end{enumerate}

A symbol \(n\ast\) adjacent to a vertex denotes the multiplicity of a batch
of \(n\) immediate descendants sharing a common parent.
The groups of particular importance are labelled by a number in angles,
which is the identifier in the SmallGroups library
\cite{BEO1,BEO2}
of
MAGMA
\cite{MAGMA},
where we omit the orders, which are given on the left hand scale.
The transfer kernel types \(\varkappa_1\), briefly TKT
\cite[Thm.2.5, Tbl.6--7]{Ma2},
in the bottom rectangle concern all
vertices located vertically above.
The first, resp. second, component \(\tau_1(1)\), resp. \(\tau_1(2)\), of the transfer target type (TTT)
\cite[Dfn.3.3, p.288]{Ma6}
in the left rectangle
concerns vertices on the same horizontal level with defect \(k(\mathfrak{G})=0\).
The periodicity with length \(2\) of branches,
\(\mathcal{B}(j)\simeq\mathcal{B}(j+2)\) for \(j\ge 7\),
sets in with branch \(\mathcal{B}(7)\), having a root of order \(3^7\).

The metabelian skeletons were drawn in
\cite[p.189ff]{Ne},
the complete trees were given in
\cite[Fig.4.8, p.76, and Fig.6.1, p.123]{As}.

The actual realization of a vertex as the second \(3\)-class group \(\mathfrak{G}=G_3^2(K)\)
of a complex, resp. real, quadratic number field \(K=\mathbb{Q}(\sqrt{d})\) of type \((3,3)\)
with discriminant \(-10^8<d<10^9\)
is indicated by an underlined boldface signed discriminant (minimal in absolute value)
adjacent to an oval surrounding the vertex.

The realization of mainline vertices
with TKT \(\mathrm{c}.18\) and \(\mathrm{c}.21\)
as \(\mathfrak{G}=G_3^2(K)\),
starting with \(\langle 729,49\rangle\) on the tree \(\mathcal{T}^2(\langle 243,6\rangle)\)
and with \(\langle 729,54\rangle\) on the tree \(\mathcal{T}^2(\langle 243,8\rangle)\),
disproved the \textit{strict} leaf conjecture
(that a second \(3\)-class group must be a leaf of the metabelian skeleton of a coclass graph)
but it is no violation of the \textit{weak} leaf conjecture
\cite[Cnj.3.1, p.423]{Ma4},
since these vertices do not possess metabelian immediate descendants
of the same TKT with higher defect \(k(\mathfrak{G})\) of commutativity.

\newpage


{\tiny

\begin{figure}[ht]
\caption{\(3\)-groups \(G\) represented by vertices of the coclass tree \(\mathcal{T}^2(\langle 243,6\rangle)\)}
\label{fig:TreeOverviewQ}


\setlength{\unitlength}{0.9cm}
\begin{picture}(17,22.5)(-7,-21.5)

\put(-8,0.5){\makebox(0,0)[cb]{order \(3^n\)}}
\put(-8,0){\line(0,-1){20}}
\multiput(-8.1,0)(0,-2){11}{\line(1,0){0.2}}
\put(-8.2,0){\makebox(0,0)[rc]{\(243\)}}
\put(-7.8,0){\makebox(0,0)[lc]{\(3^5\)}}
\put(-8.2,-2){\makebox(0,0)[rc]{\(729\)}}
\put(-7.8,-2){\makebox(0,0)[lc]{\(3^6\)}}
\put(-8.2,-4){\makebox(0,0)[rc]{\(2\,187\)}}
\put(-7.8,-4){\makebox(0,0)[lc]{\(3^7\)}}
\put(-8.2,-6){\makebox(0,0)[rc]{\(6\,561\)}}
\put(-7.8,-6){\makebox(0,0)[lc]{\(3^8\)}}
\put(-8.2,-8){\makebox(0,0)[rc]{\(19\,683\)}}
\put(-7.8,-8){\makebox(0,0)[lc]{\(3^9\)}}
\put(-8.2,-10){\makebox(0,0)[rc]{\(59\,049\)}}
\put(-7.8,-10){\makebox(0,0)[lc]{\(3^{10}\)}}
\put(-8.2,-12){\makebox(0,0)[rc]{\(177\,147\)}}
\put(-7.8,-12){\makebox(0,0)[lc]{\(3^{11}\)}}
\put(-8.2,-14){\makebox(0,0)[rc]{\(531\,441\)}}
\put(-7.8,-14){\makebox(0,0)[lc]{\(3^{12}\)}}
\put(-8.2,-16){\makebox(0,0)[rc]{\(1\,594\,323\)}}
\put(-7.8,-16){\makebox(0,0)[lc]{\(3^{13}\)}}
\put(-8.2,-18){\makebox(0,0)[rc]{\(4\,782\,969\)}}
\put(-7.8,-18){\makebox(0,0)[lc]{\(3^{14}\)}}
\put(-8.2,-20){\makebox(0,0)[rc]{\(14\,348\,907\)}}
\put(-7.8,-20){\makebox(0,0)[lc]{\(3^{15}\)}}
\put(-8,-20){\vector(0,-1){2}}

\put(-6,0.5){\makebox(0,0)[cb]{\(\tau_1(1)=\)}}
\put(-6,0){\makebox(0,0)[cc]{\((21)\)}}
\put(-6,-2){\makebox(0,0)[cc]{\((2^2)\)}}
\put(-6,-4){\makebox(0,0)[cc]{\((32)\)}}
\put(-6,-6){\makebox(0,0)[cc]{\((3^2)\)}}
\put(-6,-8){\makebox(0,0)[cc]{\((43)\)}}
\put(-6,-10){\makebox(0,0)[cc]{\((4^2)\)}}
\put(-6,-12){\makebox(0,0)[cc]{\((54)\)}}
\put(-6,-14){\makebox(0,0)[cc]{\((5^2)\)}}
\put(-6,-16){\makebox(0,0)[cc]{\((65)\)}}
\put(-6,-18){\makebox(0,0)[cc]{\((6^2)\)}}
\put(-6,-20){\makebox(0,0)[cc]{\((76)\)}}
\put(-6,-21){\makebox(0,0)[cc]{\textbf{TTT}}}
\put(-6.5,-21.2){\framebox(1,22){}}

\put(7.6,-7){\vector(0,1){3}}
\put(7.8,-7){\makebox(0,0)[lc]{depth \(3\)}}
\put(7.6,-7){\vector(0,-1){3}}

\put(-3.1,-8){\vector(0,1){2}}
\put(-3.3,-8){\makebox(0,0)[rc]{period length \(2\)}}
\put(-3.1,-8){\vector(0,-1){2}}

\put(0.7,-2){\makebox(0,0)[lc]{bifurcation from}}
\put(0.7,-2.3){\makebox(0,0)[lc]{\(\mathcal{G}(3,2)\) to \(\mathcal{G}(3,3)\)}}

\multiput(0,0)(0,-2){10}{\circle*{0.2}}
\multiput(0,0)(0,-2){9}{\line(0,-1){2}}
\multiput(-1,-2)(0,-2){10}{\circle*{0.2}}
\multiput(-2,-2)(0,-2){10}{\circle*{0.2}}
\multiput(1.95,-4.05)(0,-2){9}{\framebox(0.1,0.1){}}
\multiput(3,-2)(0,-2){10}{\circle*{0.2}}
\multiput(0,0)(0,-2){10}{\line(-1,-2){1}}
\multiput(0,0)(0,-2){10}{\line(-1,-1){2}}
\multiput(0,-2)(0,-2){9}{\line(1,-1){2}}
\multiput(0,0)(0,-2){10}{\line(3,-2){3}}
\multiput(-3.05,-4.05)(-1,0){2}{\framebox(0.1,0.1){}}
\multiput(3.95,-4.05)(0,-2){9}{\framebox(0.1,0.1){}}
\multiput(5,-4)(0,-2){9}{\circle*{0.1}}
\multiput(6,-4)(0,-2){9}{\circle*{0.1}}
\multiput(-1,-2)(-1,0){2}{\line(-1,-1){2}}
\multiput(3,-2)(0,-2){9}{\line(1,-2){1}}
\multiput(3,-2)(0,-2){9}{\line(1,-1){2}}
\multiput(3,-2)(0,-2){9}{\line(3,-2){3}}
\multiput(6.95,-6.05)(0,-2){8}{\framebox(0.1,0.1){}}
\multiput(6,-4)(0,-2){8}{\line(1,-2){1}}

\put(2,-0.5){\makebox(0,0)[lc]{branch}}
\put(2,-0.8){\makebox(0,0)[lc]{\(\mathcal{B}(5)\)}}
\put(2,-2.8){\makebox(0,0)[lc]{\(\mathcal{B}(6)\)}}
\put(2,-4.8){\makebox(0,0)[lc]{\(\mathcal{B}(7)\)}}
\put(2,-6.8){\makebox(0,0)[lc]{\(\mathcal{B}(8)\)}}
\put(2,-8.8){\makebox(0,0)[lc]{\(\mathcal{B}(9)\)}}
\put(2,-10.8){\makebox(0,0)[lc]{\(\mathcal{B}(10)\)}}
\put(2,-12.8){\makebox(0,0)[lc]{\(\mathcal{B}(11)\)}}
\put(2,-14.8){\makebox(0,0)[lc]{\(\mathcal{B}(12)\)}}
\put(2,-16.8){\makebox(0,0)[lc]{\(\mathcal{B}(13)\)}}
\put(2,-18.8){\makebox(0,0)[lc]{\(\mathcal{B}(14)\)}}

\put(-0.1,0.3){\makebox(0,0)[rc]{\(\langle 6\rangle\)}}
\put(-2.1,-1.8){\makebox(0,0)[rc]{\(\langle 50\rangle\)}}
\put(-1.1,-1.8){\makebox(0,0)[rc]{\(\langle 51\rangle\)}}
\put(0.1,-1.8){\makebox(0,0)[lc]{\(\langle 49\rangle\)}}
\put(3.1,-1.8){\makebox(0,0)[lc]{\(\langle 48\rangle\)}}
\put(-4.1,-3.5){\makebox(0,0)[cc]{\(\langle 292\rangle\)}}
\put(-3.1,-3.5){\makebox(0,0)[cc]{\(\langle 293\rangle\)}}
\put(-2.1,-3.3){\makebox(0,0)[cc]{\(\langle 289\rangle\)}}
\put(-2.1,-3.5){\makebox(0,0)[cc]{\(\langle 290\rangle\)}}
\put(-1.1,-3.5){\makebox(0,0)[cc]{\(\langle 288\rangle\)}}
\put(0.1,-3.5){\makebox(0,0)[lc]{\(\langle 285\rangle\)}}
\put(2.2,-3.3){\makebox(0,0)[cc]{\(\langle 284\rangle\)}}
\put(2.2,-3.5){\makebox(0,0)[cc]{\(\langle 291\rangle\)}}
\put(3.2,-3.3){\makebox(0,0)[cc]{\(\langle 286\rangle\)}}
\put(3.2,-3.5){\makebox(0,0)[cc]{\(\langle 287\rangle\)}}
\put(4.2,-3.3){\makebox(0,0)[cc]{\(\langle 276\rangle\)}}
\put(4.2,-3.5){\makebox(0,0)[cc]{\(\langle 283\rangle\)}}
\put(5.2,-3.1){\makebox(0,0)[cc]{\(\langle 280\rangle\)}}
\put(5.2,-3.3){\makebox(0,0)[cc]{\(\langle 281\rangle\)}}
\put(5.2,-3.5){\makebox(0,0)[cc]{\(\langle 282\rangle\)}}
\put(6.2,-3.1){\makebox(0,0)[cc]{\(\langle 277\rangle\)}}
\put(6.2,-3.3){\makebox(0,0)[cc]{\(\langle 278\rangle\)}}
\put(6.2,-3.5){\makebox(0,0)[cc]{\(\langle 279\rangle\)}}

\put(2.1,-3.8){\makebox(0,0)[lc]{\(*2\)}}
\multiput(-2.1,-3.8)(0,-4){5}{\makebox(0,0)[rc]{\(2*\)}}
\multiput(3.1,-3.8)(0,-4){5}{\makebox(0,0)[lc]{\(*2\)}}
\put(4.1,-3.8){\makebox(0,0)[lc]{\(*2\)}}
\multiput(5.1,-5.8)(0,-4){4}{\makebox(0,0)[lc]{\(*2\)}}
\multiput(5.5,-5.3)(0,-4){4}{\makebox(0,0)[lc]{\(\#2\)}}
\multiput(6.1,-5.8)(0,-4){4}{\makebox(0,0)[lc]{\(*2\)}}
\multiput(5.1,-3.8)(0,-4){5}{\makebox(0,0)[lc]{\(*3\)}}
\multiput(6.1,-3.8)(0,-4){5}{\makebox(0,0)[lc]{\(*3\)}}
\multiput(7.1,-5.8)(0,-2){4}{\makebox(0,0)[lc]{\(*3\)}}

\put(-3,-21){\makebox(0,0)[cc]{\textbf{TKT}}}
\put(-2,-21){\makebox(0,0)[cc]{E.14}}
\put(-1,-21){\makebox(0,0)[cc]{E.6}}
\put(0,-21){\makebox(0,0)[cc]{c.18}}
\put(2,-21){\makebox(0,0)[cc]{c.18}}
\put(3.1,-21){\makebox(0,0)[cc]{H.4}}
\put(4,-21){\makebox(0,0)[cc]{H.4}}
\put(5,-21){\makebox(0,0)[cc]{H.4}}
\put(6,-21){\makebox(0,0)[cc]{H.4}}
\put(7,-21){\makebox(0,0)[cc]{H.4}}
\put(-3,-21.5){\makebox(0,0)[cc]{\(\varkappa_1=\)}}
\put(-2,-21.5){\makebox(0,0)[cc]{\((3122)\)}}
\put(-1,-21.5){\makebox(0,0)[cc]{\((1122)\)}}
\put(0,-21.5){\makebox(0,0)[cc]{\((0122)\)}}
\put(2,-21.5){\makebox(0,0)[cc]{\((0122)\)}}
\put(3.1,-21.5){\makebox(0,0)[cc]{\((2122)\)}}
\put(4,-21.5){\makebox(0,0)[cc]{\((2122)\)}}
\put(5,-21.5){\makebox(0,0)[cc]{\((2122)\)}}
\put(6,-21.5){\makebox(0,0)[cc]{\((2122)\)}}
\put(7,-21.5){\makebox(0,0)[cc]{\((2122)\)}}
\put(-3.8,-21.7){\framebox(11.6,1){}}

\put(0,-18){\vector(0,-1){2}}
\put(0.2,-19.4){\makebox(0,0)[lc]{infinite}}
\put(0.2,-19.9){\makebox(0,0)[lc]{mainline}}
\put(1.8,-20.4){\makebox(0,0)[rc]{\(\mathcal{T}^2(\langle 243,6\rangle)\)}}


\multiput(0,-2)(0,-4){3}{\oval(1,1)}
\put(0.1,-1.3){\makebox(0,0)[lc]{\underbar{\textbf{+534\,824}}}}
\put(0.1,-5.3){\makebox(0,0)[lc]{\underbar{\textbf{+13\,714\,789}}}}
\put(0.1,-9.3){\makebox(0,0)[lc]{\underbar{\textbf{+174\,458\,681}}}}

\multiput(-1,-4)(0,-4){4}{\oval(1,1)}
\put(-1,-4.8){\makebox(0,0)[lc]{\underbar{\textbf{-15\,544}}}}
\put(-1,-8.8){\makebox(0,0)[lc]{\underbar{\textbf{-268\,040}}}}
\put(-1,-12.8){\makebox(0,0)[lc]{\underbar{\textbf{-1\,062\,708}}}}
\put(-1,-16.8){\makebox(0,0)[lc]{\underbar{\textbf{-27\,629\,107}}}}
\multiput(-2,-4)(0,-4){4}{\oval(1,1)}
\put(-2,-4.8){\makebox(0,0)[rc]{\underbar{\textbf{-16\,627}}}}
\put(-2,-8.8){\makebox(0,0)[rc]{\underbar{\textbf{-262\,744}}}}
\put(-2,-12.8){\makebox(0,0)[rc]{\underbar{\textbf{-4\,776\,071}}}}
\put(-2,-16.8){\makebox(0,0)[rc]{\underbar{\textbf{-40\,059\,363}}}}
\multiput(6,-6)(0,-4){4}{\oval(1,1)}
\put(6,-6.8){\makebox(0,0)[rc]{\underbar{\textbf{-21\,668}}}}
\put(6,-10.8){\makebox(0,0)[rc]{\underbar{\textbf{-446\,788}}}}
\put(6,-14.8){\makebox(0,0)[rc]{\underbar{\textbf{-3\,843\,907}}}}
\put(6,-18.8){\makebox(0,0)[rc]{\underbar{\textbf{-52\,505\,588}}}}

\end{picture}

\end{figure}

}

\noindent
Figure
\ref{fig:TreeOverviewQ}
visualizes \(3\)-groups which form the crucial objects in section \S\
\ref{s:c18}.

\newpage


{\tiny

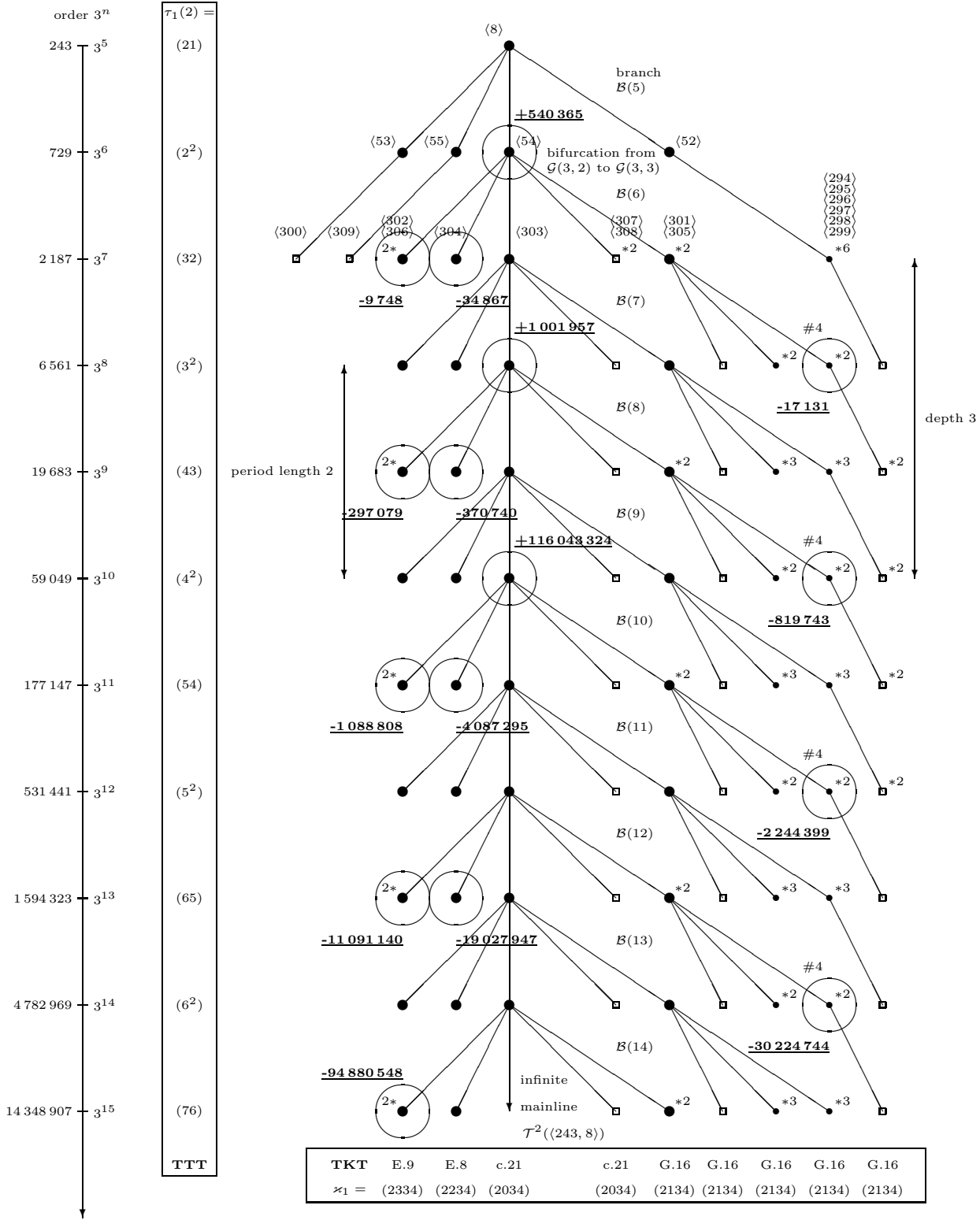
\begin{figure}[ht]
\caption{\(3\)-groups \(G\) represented by vertices of the coclass tree \(\mathcal{T}^2(\langle 243,8\rangle)\)}
\label{fig:TreeOverviewU}


\setlength{\unitlength}{0.9cm}
\begin{picture}(17,22.5)(-7,-21.5)

\put(-8,0.5){\makebox(0,0)[cb]{order \(3^n\)}}
\put(-8,0){\line(0,-1){20}}
\multiput(-8.1,0)(0,-2){11}{\line(1,0){0.2}}
\put(-8.2,0){\makebox(0,0)[rc]{\(243\)}}
\put(-7.8,0){\makebox(0,0)[lc]{\(3^5\)}}
\put(-8.2,-2){\makebox(0,0)[rc]{\(729\)}}
\put(-7.8,-2){\makebox(0,0)[lc]{\(3^6\)}}
\put(-8.2,-4){\makebox(0,0)[rc]{\(2\,187\)}}
\put(-7.8,-4){\makebox(0,0)[lc]{\(3^7\)}}
\put(-8.2,-6){\makebox(0,0)[rc]{\(6\,561\)}}
\put(-7.8,-6){\makebox(0,0)[lc]{\(3^8\)}}
\put(-8.2,-8){\makebox(0,0)[rc]{\(19\,683\)}}
\put(-7.8,-8){\makebox(0,0)[lc]{\(3^9\)}}
\put(-8.2,-10){\makebox(0,0)[rc]{\(59\,049\)}}
\put(-7.8,-10){\makebox(0,0)[lc]{\(3^{10}\)}}
\put(-8.2,-12){\makebox(0,0)[rc]{\(177\,147\)}}
\put(-7.8,-12){\makebox(0,0)[lc]{\(3^{11}\)}}
\put(-8.2,-14){\makebox(0,0)[rc]{\(531\,441\)}}
\put(-7.8,-14){\makebox(0,0)[lc]{\(3^{12}\)}}
\put(-8.2,-16){\makebox(0,0)[rc]{\(1\,594\,323\)}}
\put(-7.8,-16){\makebox(0,0)[lc]{\(3^{13}\)}}
\put(-8.2,-18){\makebox(0,0)[rc]{\(4\,782\,969\)}}
\put(-7.8,-18){\makebox(0,0)[lc]{\(3^{14}\)}}
\put(-8.2,-20){\makebox(0,0)[rc]{\(14\,348\,907\)}}
\put(-7.8,-20){\makebox(0,0)[lc]{\(3^{15}\)}}
\put(-8,-20){\vector(0,-1){2}}

\put(-6,0.5){\makebox(0,0)[cb]{\(\tau_1(2)=\)}}
\put(-6,0){\makebox(0,0)[cc]{\((21)\)}}
\put(-6,-2){\makebox(0,0)[cc]{\((2^2)\)}}
\put(-6,-4){\makebox(0,0)[cc]{\((32)\)}}
\put(-6,-6){\makebox(0,0)[cc]{\((3^2)\)}}
\put(-6,-8){\makebox(0,0)[cc]{\((43)\)}}
\put(-6,-10){\makebox(0,0)[cc]{\((4^2)\)}}
\put(-6,-12){\makebox(0,0)[cc]{\((54)\)}}
\put(-6,-14){\makebox(0,0)[cc]{\((5^2)\)}}
\put(-6,-16){\makebox(0,0)[cc]{\((65)\)}}
\put(-6,-18){\makebox(0,0)[cc]{\((6^2)\)}}
\put(-6,-20){\makebox(0,0)[cc]{\((76)\)}}
\put(-6,-21){\makebox(0,0)[cc]{\textbf{TTT}}}
\put(-6.5,-21.2){\framebox(1,22){}}

\put(7.6,-7){\vector(0,1){3}}
\put(7.8,-7){\makebox(0,0)[lc]{depth \(3\)}}
\put(7.6,-7){\vector(0,-1){3}}

\put(-3.1,-8){\vector(0,1){2}}
\put(-3.3,-8){\makebox(0,0)[rc]{period length \(2\)}}
\put(-3.1,-8){\vector(0,-1){2}}

\put(0.7,-2){\makebox(0,0)[lc]{bifurcation from}}
\put(0.7,-2.3){\makebox(0,0)[lc]{\(\mathcal{G}(3,2)\) to \(\mathcal{G}(3,3)\)}}

\multiput(0,0)(0,-2){10}{\circle*{0.2}}
\multiput(0,0)(0,-2){9}{\line(0,-1){2}}
\multiput(-1,-2)(0,-2){10}{\circle*{0.2}}
\multiput(-2,-2)(0,-2){10}{\circle*{0.2}}
\multiput(1.95,-4.05)(0,-2){9}{\framebox(0.1,0.1){}}
\multiput(3,-2)(0,-2){10}{\circle*{0.2}}
\multiput(0,0)(0,-2){10}{\line(-1,-2){1}}
\multiput(0,0)(0,-2){10}{\line(-1,-1){2}}
\multiput(0,-2)(0,-2){9}{\line(1,-1){2}}
\multiput(0,0)(0,-2){10}{\line(3,-2){3}}
\multiput(-3.05,-4.05)(-1,0){2}{\framebox(0.1,0.1){}}
\multiput(3.95,-6.05)(0,-2){8}{\framebox(0.1,0.1){}}
\multiput(5,-6)(0,-2){8}{\circle*{0.1}}
\multiput(6,-4)(0,-2){9}{\circle*{0.1}}
\multiput(-1,-2)(-1,0){2}{\line(-1,-1){2}}
\multiput(3,-4)(0,-2){8}{\line(1,-2){1}}
\multiput(3,-4)(0,-2){8}{\line(1,-1){2}}
\multiput(3,-2)(0,-2){9}{\line(3,-2){3}}
\multiput(6.95,-6.05)(0,-2){8}{\framebox(0.1,0.1){}}
\multiput(6,-4)(0,-2){8}{\line(1,-2){1}}

\put(2,-0.5){\makebox(0,0)[lc]{branch}}
\put(2,-0.8){\makebox(0,0)[lc]{\(\mathcal{B}(5)\)}}
\put(2,-2.8){\makebox(0,0)[lc]{\(\mathcal{B}(6)\)}}
\put(2,-4.8){\makebox(0,0)[lc]{\(\mathcal{B}(7)\)}}
\put(2,-6.8){\makebox(0,0)[lc]{\(\mathcal{B}(8)\)}}
\put(2,-8.8){\makebox(0,0)[lc]{\(\mathcal{B}(9)\)}}
\put(2,-10.8){\makebox(0,0)[lc]{\(\mathcal{B}(10)\)}}
\put(2,-12.8){\makebox(0,0)[lc]{\(\mathcal{B}(11)\)}}
\put(2,-14.8){\makebox(0,0)[lc]{\(\mathcal{B}(12)\)}}
\put(2,-16.8){\makebox(0,0)[lc]{\(\mathcal{B}(13)\)}}
\put(2,-18.8){\makebox(0,0)[lc]{\(\mathcal{B}(14)\)}}

\put(-0.1,0.3){\makebox(0,0)[rc]{\(\langle 8\rangle\)}}
\put(-2.1,-1.8){\makebox(0,0)[rc]{\(\langle 53\rangle\)}}
\put(-1.1,-1.8){\makebox(0,0)[rc]{\(\langle 55\rangle\)}}
\put(0.1,-1.8){\makebox(0,0)[lc]{\(\langle 54\rangle\)}}
\put(3.1,-1.8){\makebox(0,0)[lc]{\(\langle 52\rangle\)}}
\put(-4.1,-3.5){\makebox(0,0)[cc]{\(\langle 300\rangle\)}}
\put(-3.1,-3.5){\makebox(0,0)[cc]{\(\langle 309\rangle\)}}
\put(-2.1,-3.3){\makebox(0,0)[cc]{\(\langle 302\rangle\)}}
\put(-2.1,-3.5){\makebox(0,0)[cc]{\(\langle 306\rangle\)}}
\put(-1.1,-3.5){\makebox(0,0)[cc]{\(\langle 304\rangle\)}}
\put(0.1,-3.5){\makebox(0,0)[lc]{\(\langle 303\rangle\)}}
\put(2.2,-3.3){\makebox(0,0)[cc]{\(\langle 307\rangle\)}}
\put(2.2,-3.5){\makebox(0,0)[cc]{\(\langle 308\rangle\)}}
\put(3.2,-3.3){\makebox(0,0)[cc]{\(\langle 301\rangle\)}}
\put(3.2,-3.5){\makebox(0,0)[cc]{\(\langle 305\rangle\)}}
\put(6.2,-2.5){\makebox(0,0)[cc]{\(\langle 294\rangle\)}}
\put(6.2,-2.7){\makebox(0,0)[cc]{\(\langle 295\rangle\)}}
\put(6.2,-2.9){\makebox(0,0)[cc]{\(\langle 296\rangle\)}}
\put(6.2,-3.1){\makebox(0,0)[cc]{\(\langle 297\rangle\)}}
\put(6.2,-3.3){\makebox(0,0)[cc]{\(\langle 298\rangle\)}}
\put(6.2,-3.5){\makebox(0,0)[cc]{\(\langle 299\rangle\)}}

\put(2.1,-3.8){\makebox(0,0)[lc]{\(*2\)}}
\multiput(-2.1,-3.8)(0,-4){5}{\makebox(0,0)[rc]{\(2*\)}}
\multiput(3.1,-3.8)(0,-4){5}{\makebox(0,0)[lc]{\(*2\)}}
\put(6.1,-3.8){\makebox(0,0)[lc]{\(*6\)}}
\multiput(5.1,-5.8)(0,-4){4}{\makebox(0,0)[lc]{\(*2\)}}
\multiput(5.5,-5.3)(0,-4){4}{\makebox(0,0)[lc]{\(\#4\)}}
\multiput(6.1,-5.8)(0,-4){4}{\makebox(0,0)[lc]{\(*2\)}}
\multiput(5.1,-7.8)(0,-4){4}{\makebox(0,0)[lc]{\(*3\)}}
\multiput(6.1,-7.8)(0,-4){4}{\makebox(0,0)[lc]{\(*3\)}}
\multiput(7.1,-7.8)(0,-2){4}{\makebox(0,0)[lc]{\(*2\)}}

\put(-3,-21){\makebox(0,0)[cc]{\textbf{TKT}}}
\put(-2,-21){\makebox(0,0)[cc]{E.9}}
\put(-1,-21){\makebox(0,0)[cc]{E.8}}
\put(0,-21){\makebox(0,0)[cc]{c.21}}
\put(2,-21){\makebox(0,0)[cc]{c.21}}
\put(3.1,-21){\makebox(0,0)[cc]{G.16}}
\put(4,-21){\makebox(0,0)[cc]{G.16}}
\put(5,-21){\makebox(0,0)[cc]{G.16}}
\put(6,-21){\makebox(0,0)[cc]{G.16}}
\put(7,-21){\makebox(0,0)[cc]{G.16}}
\put(-3,-21.5){\makebox(0,0)[cc]{\(\varkappa_1=\)}}
\put(-2,-21.5){\makebox(0,0)[cc]{\((2334)\)}}
\put(-1,-21.5){\makebox(0,0)[cc]{\((2234)\)}}
\put(0,-21.5){\makebox(0,0)[cc]{\((2034)\)}}
\put(2,-21.5){\makebox(0,0)[cc]{\((2034)\)}}
\put(3.1,-21.5){\makebox(0,0)[cc]{\((2134)\)}}
\put(4,-21.5){\makebox(0,0)[cc]{\((2134)\)}}
\put(5,-21.5){\makebox(0,0)[cc]{\((2134)\)}}
\put(6,-21.5){\makebox(0,0)[cc]{\((2134)\)}}
\put(7,-21.5){\makebox(0,0)[cc]{\((2134)\)}}
\put(-3.8,-21.7){\framebox(11.6,1){}}

\put(0,-18){\vector(0,-1){2}}
\put(0.2,-19.4){\makebox(0,0)[lc]{infinite}}
\put(0.2,-19.9){\makebox(0,0)[lc]{mainline}}
\put(1.8,-20.4){\makebox(0,0)[rc]{\(\mathcal{T}^2(\langle 243,8\rangle)\)}}


\multiput(0,-2)(0,-4){3}{\oval(1,1)}
\put(0.1,-1.3){\makebox(0,0)[lc]{\underbar{\textbf{+540\,365}}}}
\put(0.1,-5.3){\makebox(0,0)[lc]{\underbar{\textbf{+1\,001\,957}}}}
\put(0.1,-9.3){\makebox(0,0)[lc]{\underbar{\textbf{+116\,043\,324}}}}

\multiput(-1,-4)(0,-4){4}{\oval(1,1)}
\put(-1,-4.8){\makebox(0,0)[lc]{\underbar{\textbf{-34\,867}}}}
\put(-1,-8.8){\makebox(0,0)[lc]{\underbar{\textbf{-370\,740}}}}
\put(-1,-12.8){\makebox(0,0)[lc]{\underbar{\textbf{-4\,087\,295}}}}
\put(-1,-16.8){\makebox(0,0)[lc]{\underbar{\textbf{-19\,027\,947}}}}
\multiput(-2,-4)(0,-4){5}{\oval(1,1)}
\put(-2,-4.8){\makebox(0,0)[rc]{\underbar{\textbf{-9\,748}}}}
\put(-2,-8.8){\makebox(0,0)[rc]{\underbar{\textbf{-297\,079}}}}
\put(-2,-12.8){\makebox(0,0)[rc]{\underbar{\textbf{-1\,088\,808}}}}
\put(-2,-16.8){\makebox(0,0)[rc]{\underbar{\textbf{-11\,091\,140}}}}
\put(-2,-19.3){\makebox(0,0)[rc]{\underbar{\textbf{-94\,880\,548}}}}
\multiput(6,-6)(0,-4){4}{\oval(1,1)}
\put(6,-6.8){\makebox(0,0)[rc]{\underbar{\textbf{-17\,131}}}}
\put(6,-10.8){\makebox(0,0)[rc]{\underbar{\textbf{-819\,743}}}}
\put(6,-14.8){\makebox(0,0)[rc]{\underbar{\textbf{-2\,244\,399}}}}
\put(6,-18.8){\makebox(0,0)[rc]{\underbar{\textbf{-30\,224\,744}}}}

\end{picture}

\end{figure}

}

\noindent
Figure
\ref{fig:TreeOverviewU}
visualizes \(3\)-groups which form the crucial objects in sections \S\
\ref{s:c21}.

\newpage


\section{Pattern Recognition via Artin Transfers}
\label{s:ArtinPattern}

Let \(p\) be a fixed prime number
and \(k\) be a number field with \(p\)-class group \(\mathrm{Cl}_p(k)\)
of order \(p^v\), where \(v\ge 0\) denotes a non-negative integer.
According to the Artin reciprocity law of class field theory
\cite{Ar1},
\(\mathrm{Cl}_p(k)\) is isomorphic to the Galois group
\(\mathrm{G}_p^1(k):=\mathrm{Gal}(\mathrm{F}_p^1(k)\mid k)\)
of the maximal abelian unramified \(p\)-extension,
that is the (first) Hilbert \(p\)-class field \(\mathrm{F}_p^1(k)\), of \(k\).

In section \S\
\ref{s:Intro},
we defined the \(p\)-class tower group (briefly \(p\)-\textit{tower group}) of \(k\)
as the Galois group \(G=\mathrm{G}_p^\infty(k):=\mathrm{Gal}(\mathrm{F}_p^\infty(k)\mid k)\)
of the maximal unramified pro-\(p\) extension \(\mathrm{F}_p^\infty(k)\) of \(k\).
The fixed field of the commutator subgroup \(G^\prime\)
(the minimal subgroup with abelian quotient) in \(\mathrm{F}_p^\infty(k)\)
is the maximal abelian unramified \(p\)-extension \(\mathrm{F}_p^1(k)\) of \(k\)
with Galois group
\[\mathrm{Gal}(\mathrm{F}_p^1(k)\mid k)\simeq
\mathrm{Gal}(\mathrm{F}_p^\infty(k)\mid k)/\mathrm{Gal}(\mathrm{F}_p^\infty(k)\mid\mathrm{F}_p^1(k))=G/G^\prime.\]
Thus, the \(p\)-class group \(\mathrm{Cl}_p(k)\) of \(k\) is also isomorphic to
the abelianization \(G/G^\prime\) of the \(p\)-tower group \(G\) of \(k\).
Consequently, the derived subgroup \(G^\prime\) is a closed (and open) subgroup of finite index
\((G:G^\prime)=p^v\) in the topological pro-\(p\) group \(G\),
and there exist \(v+1\) layers
\[\lbrace G^\prime\le H\le G\mid (G:H)=p^n\rbrace,\text{ for }0\le n\le v,\]
of intermediate normal subgroups \(H\unlhd G\)
between \(G\) and its commutator subgroup \(G^\prime\).
For each of them, we denote by
\(T_{G,H}:G\to H/H^\prime\)
the Artin transfer homomorphism from \(G\) to \(H\)
\cite{Ar2}.



\begin{definition}
\label{dfn:LayerGT}
For each integer \(0\le n\le v\),
the system

\begin{equation}
\label{eqn:LayerGT}
\mathrm{Lyr}_n(G):=\lbrace G^\prime\le H\le G\mid (G:H)=p^n\rbrace
\end{equation}

\noindent
of intermediate groups \(G^\prime\le H\le G\)
with index \((G:H)=p^n\)
is called the \(n\)-\textit{th layer} of normal subgroups of \(G\) with abelian quotients \(G/H\).
In particular,
for \(n=0\), \(G\) forms the \textit{top layer}
\(\mathrm{Lyr}_0(G)=\lbrace G\rbrace\),
and for \(n=v\), \(G^\prime\) constitutes the \textit{bottom layer}
\(\mathrm{Lyr}_v(K)=\lbrace G^\prime\rbrace\).
\end{definition}

Since the abelianization \(H^{\mathrm{ab}}=H/H^\prime\) forms the target of the
Artin transfer homomorphism \(T_{G,H}:\,G\to H/H^\prime\) from \(G\) to \(H\),
we introduced a preliminary instance of the following terminology in
\cite[Dfn.1.1, p.403]{Ma4}.



\begin{definition}
\label{dfn:TTTandTKT}
For each integer \(0\le n\le v\),
the family

\begin{equation}
\label{eqn:TTT}
\tau_n(G):=(H/H^\prime)_{H\in\mathrm{Lyr}_n(G)}
\end{equation}

\noindent
of abelianizations
is called the \(n\)-th layer of the multi-layered \textit{transfer target type} (TTT)
of \(G\),

\begin{equation}
\label{eqn:TTTGT}
\tau(G):=\lbrack\tau_0(G);\ldots;\tau_v(G)\rbrack.
\end{equation}

\noindent
For each integer \(0\le n\le v\),
the family

\begin{equation}
\label{eqn:TKT}
\varkappa_n(G):=(\ker(T_{G,H}))_{H\in\mathrm{Lyr}_n(G)}
\end{equation}

\noindent
of transfer kernels
is called the \(n\)-th layer of the multi-layered \textit{transfer kernel type} (TKT)
of \(G\),

\begin{equation}
\label{eqn:TKTGT}
\varkappa(G):=\lbrack\varkappa_0(G);\ldots;\varkappa_v(G)\rbrack.
\end{equation}

\end{definition}

\begin{definition}
\label{dfn:ArtinPattern}
By the \textit{Artin pattern} of \(G\) we understand the pair

\begin{equation}
\label{eqn:ArtinPattern}
\mathrm{AP}(G):=(\tau(G);\varkappa(G))
\end{equation}

\noindent
consisting of the multi-layered TTT \(\tau(G)\)
and the multi-layered TKT \(\varkappa(G)\)
of \(G\).

\end{definition}



\begin{remark}
\label{rmk:TTT}
Suppose that \(0<n<v\).
If an ordering is defined for the elements of \(\mathrm{Lyr}_n(G)\),
then the same ordering is applied to the members of the layer \(\tau_n(G)\)
and the TTT layer is called \textit{ordered}.
Otherwise, the TTT layer is called \textit{unordered} or \textit{accumulated},
since equal components are collected in powers with formal exponents denoting iteration.
The same considerations apply to the TKT.
\end{remark}

Since it is increasingly difficult to compute the structure of the \(p\)-class groups \(\mathrm{Cl}_p(L)\)
of extension fields \(L\) with \(\lbrack L:k\rbrack=p^n\) in higher layers with \(n\ge 2\),
it is frequently sufficient to make use of information in the first layer only,
that is the layer of subgroups with index \(p\).
Therefore, Boston, Bush and Hajir
\cite{BBH}
invented the following \textit{first order approximation} of the TTT,
and we add a supplementary notion for the TKT.



\begin{definition}
\label{dfn:IPAD}
The restriction

\begin{equation}
\label{eqn:IPAD1}
\begin{aligned}
\tau^{(1)}(G) &:=\lbrack\tau_0(G);\tau_1(G)\rbrack,\text{ resp.}\\
\varkappa^{(1)}(G) &:=\lbrack\varkappa_0(G);\varkappa_1(G)\rbrack,
\end{aligned}
\end{equation}

\noindent
of the TTT \(\tau(G)\), resp. TKT \(\varkappa(G)\), to the zeroth and first layer
is called the \textit{index-\(p\) abelianization data} (IPAD),
resp. \textit{index-\(p\) obstruction data} (IPOD), of \(G\),
since the kernel is an obstruction for the injectivity of the
induced Artin transfer homomorphism \(\tilde{T}_{G,H}:\,G/G^\prime\to H/H^\prime\),
and thus also for an embedding of the \(p\)-class group of the base field \(K\)
into the \(p\)-class group of an extension \(L\).
\end{definition}

So, the complete TTT is an extension of the IPAD.
However, there also exists another extension of the IPAD
which is not covered by the TTT.
It is constructed from the usual IPAD \(\lbrack\tau_0(G);\tau_1(G)\rbrack\) of \(G\),
firstly, by observing that \(\tau_1(G)=(H/H^\prime)_{H\in\mathrm{Lyr}_1(G)}\)
can be viewed as \(\tau_1(G)=(\tau_0(H))_{H\in\mathrm{Lyr}_1(G)}\) and,
secondly, by extending each \(\tau_0(H)\) to
the IPAD \(\lbrack\tau_0(H);\tau_1(H)\rbrack\) of \(H\).
The same ideas can be applied to the TKT and IPOD.



\begin{definition}
\label{dfn:IteratedIPAD}
The family

\begin{equation}
\label{eqn:IPAD2}
\begin{aligned}
\tau^{(2)}(G) &:=\lbrack\tau_0(G);(\lbrack\tau_0(H);\tau_1(H)\rbrack)_{H\in\mathrm{Lyr}_1(G)}\rbrack,\text{ resp.}\\
\varkappa^{(2)}(G) &:=\lbrack\varkappa_0(G);(\lbrack\varkappa_0(H);\varkappa_1(H)\rbrack)_{H\in\mathrm{Lyr}_1(G)}\rbrack,
\end{aligned}
\end{equation}

\noindent
is called the \textit{iterated} IPAD, resp. IPOD, \textit{of second order} of \(G\).
\end{definition}



\section{Relation rank of the \(p\)-class tower group}
\label{s:RelRank}

\noindent
There is a fatal misprint in both,
the Russian original of the Shafarevich paper
\cite{Sh}
of 1964,
and in the English translation of 1966,
where the Dirichlet unit rank \(r\) is missing
in formula (\(18^\prime\)).
Therefore, we prove the following correction.

\begin{theorem}
\label{thm:Shafarevich}
(Corrected version of Theorem 6 by I.R. Shafarevich \cite{Sh})\\
Let \(p\) be a prime number and
denote by \(\zeta\) a primitive \(p\)th root of unity. 
Let \(k\) be a number field with signature \((r_1,r_2)\)
and torsionfree Dirichlet unit rank \(r=r_1+r_2-1\),
and let \(S\) be a finite set of non-archimedean or real archimedean places of \(k\).
Assume that no place in \(S\) divides \(p\).\\
Then the relation rank
\(d_2(G_S)=\dim_{\mathbb{F}_p}\mathrm{H}^2(G_S,\mathbb{F}_p)\)
of the Galois group \(G_S=\mathrm{Gal}(k_S\vert k)\)
of the maximal pro-\(\ell\) extension \(k_S\) of \(k\)
which is unramified outside of \(S\)
is bounded from above by 

\begin{equation}
\label{eqn:Shafarevich}
d_2(G_S)\le
\begin{cases}
d_1(G_S)+r   & \text{ if } S\ne\emptyset \text{ or }  \zeta\notin k,\\
d_1(G_S)+r+1 & \text{ if } S=\emptyset   \text{ and } \zeta\in k,
\end{cases}
\end{equation}

\noindent
where
\(d_1(G_S)=\dim_{\mathbb{F}_p}\mathrm{H}^1(G_S,\mathbb{F}_p)\)
denotes the generator rank of \(G_S\).
\end{theorem}

\begin{proof}
It suffices to prove the corrected statement
for the case \(S=\emptyset\) and \(\zeta\in k\).\\
Let \(J\), resp. \(U_S\), be
the group of id\`eles, resp. the subgroup of \(S\)-unit id\`eles, of \(k\),
and put
\(\sigma(S)=\dim_{\mathbb{F}_p}V_S/(k^\times)^p\),
where \(V_S=U_SJ^p\cap k^\times\).
We start with the second part of
\cite[\S\ 4, Thm.5, p.137]{Sh}
which states that the relation rank of \(G_S\) is bounded from above by
\(d_2(G_S)\le\sigma(S)\)
if \(S=\emptyset\) and \(\zeta\in k\).

According to
\cite[\S\ 1, Thm.1, p.130]{Sh},
the generator rank of \(G_S\) is given by
\(d_1(G_S)=t(S)+\lambda(S)+\sigma(S)-r-\delta\),
where
\(t(S)=\#\lbrace v\in S\mid\zeta\in k_v\rbrace\),
\(\lambda(S)=\sum_{v\in S,v\vert p}\,\lbrack k_v:\mathbb{Q}_v\rbrack\),
\(\delta=0\) if \(\zeta\notin k\) and
\(\delta=1\) if \(\zeta\in k\).
Our assumptions \(S=\emptyset\) and \(\zeta\in k\) imply that
\(t(S)=0\), \(\lambda(S)=0\) and \(\delta=1\).

Together, we obtain
\(d_2(G_S)-d_1(G_S)\le\sigma(S)-(t(S)+\lambda(S)+\sigma(S)-r-\delta)
=-t(S)-\lambda(S)+r+\delta
=-0-0+r+1
=r+1\),
that is
\(d_2(G_S)\le d_1(G_S)+r+1\).
\end{proof}

\begin{example}
\label{exm:Shafarevich}

The first examples of fields, where a violation of the Shafarevich Theorem
in the misprinted version occurred, have been found by Azizi, Zekhnini and Taous
\cite{AZT}:
The torsionfree unit rank of any bicyclic biquadratic field
\(k=\mathbb{Q}(\sqrt{-1},\sqrt{d})\) (\(d>1\) squarefree)
with signature \((r_1,r_2)=(0,2)\) is \(r=r_1+r_2-1=1\).
The particular fields with radicand \(d=p_1p_2q\),
where \(p_1\equiv 1\pmod{8}\), \(p_2\equiv 5\pmod{8}\) and
\(q\equiv 3\pmod{4}\) are prime numbers such that
\(\left(\frac{p_1}{p_2}\right)=-1\), \(\left(\frac{p_1}{q}\right)=-1\) and \(\left(\frac{p_2}{q}\right)=-1\),
have a \(2\)-class group \(\mathrm{Cl}_2(k)\) of type \((2,2,2)\),
a \(2\)-class tower of length \(\ell_2(k)=2\)
\cite{AZT},
and one of the \(2\)-tower groups
\(\langle 32,35\rangle\), \(\langle 64,181\rangle\), \(\langle 128,984\rangle\), etc.
\cite[Fig.2, p.752]{Ma7},
which all have \(p\)-multiplicator rank \(\mu=5\)
\cite{MAGMA}.
Since the relation rank
\(d_2(G)=\dim_{\mathbb{F}_p}\mathrm{H}^2(G,\mathbb{F}_p)\)
of a finite \(p\)-group \(G\) is bounded from below
by its \(p\)-multiplicator rank \(\mu(G)\)
\cite[\S\ 14, p.178]{Ma5},
our corrected formula admits the conclusion that the relation rank \(d_2(G)\ge\mu(G)=5\)
of the \(2\)-tower groups
\(\langle 32,35\rangle\), \(\langle 64,181\rangle\), \(\langle 128,984\rangle\), etc.
is exactly equal to \(d_1(G)+r+1=3+1+1=5\),
whereas the misprinted formula yields the contradiction \(d_2(G)\le d_1(G)+1=3+1=4\).
In this case, the two inequalities \(\mu(G)\le d_2(G)\le d_1(G)+r+1\)
degenerate to consecutive equations \(\mu(G)=d_2(G)=d_1(G)+r+1\).

\end{example}



\bigskip
Slightly different from our definitions in
\cite[Dfn.21.2, p.187]{Ma5}
and
\cite[Dfn.6.1, p.301]{Ma6},
we now redefine the concept of a cover.

\begin{definition}
\label{dfn:Cover}
Let \(\mathfrak{G}\) be a finite \textit{metabelian} \(p\)-group.

\begin{enumerate}

\item
By the \textit{cover} \(\mathrm{cov}(\mathfrak{G})\) of \(\mathfrak{G}\)
we understand the set of all (isomorphism classes of) finite \(p\)-groups \(H\)
whose second derived quotient \(H/H^{\prime\prime}\) is isomorphic to \(\mathfrak{G}\):

\begin{equation}
\label{eqn:Cover}
\mathrm{cov}(\mathfrak{G}):=\lbrace H\mid \lvert H\rvert<\infty,\ H/H^{\prime\prime}\simeq\mathfrak{G}\rbrace.
\end{equation}

\item
For any nonnegative integer \(n\ge 0\),
let the \textit{\(n\)-cover} of \(\mathfrak{G}\) be defined by

\begin{equation}
\label{eqn:nCover}
\mathrm{cov}_n(\mathfrak{G}):=\lbrace H\in\mathrm{cov}(\mathfrak{G})\mid 0\le d_2(H)-d_1(H)\le n\rbrace.
\end{equation}

\end{enumerate}

\end{definition}

\begin{remark}
\label{rmk:Cover}

\begin{enumerate}

\item
In contrast to our earlier definitions,
the cover \(\mathrm{cov}(\mathfrak{G})\) in the new sense cannot be empty,
since it certainly contains (the isomorphism class of) the group \(\mathfrak{G}\) itself
as the unique metabelian element.

\item
By the Burnside basis theorem, the defining conditions
\(0\le d_2(H)-d_1(H)\le n\)
for the \(n\)-cover of \(\mathfrak{G}\) could also be replaced by
\(d_1(\mathfrak{G})\le d_2(H)\le d_1(\mathfrak{G})+n\),
since \(d_1(H)=d_1(\mathfrak{G})\).

\end{enumerate}

\end{remark}

\begin{definition}
\label{dfn:ShafarevichCover}

Let \(k\) be a number field with \(p\)-class rank \(\rho\),
torsionfree Dirichlet unit rank \(r\), and
second \(p\)-class group \(\mathfrak{G}=\mathrm{G}_p^2(k)\).
By the \textit{Shafarevich cover} \(\mathrm{cov}_k(\mathfrak{G})\) of \(\mathfrak{G}\)
with respect to \(k\) (or to a class of similar number fields)
we understand the \(n\)-cover of \(\mathfrak{G}\) with
\(n=r+1\) if \(k\) contains a primitive \(p\)th root of unity, and \(n=r\) otherwise.
The Shafarevich cover can also be defined directly by

\begin{equation}
\label{eqn:ShafarevichCover}
\mathrm{cov}_k(\mathfrak{G}):=\lbrace H\in\mathrm{cov}(\mathfrak{G})\mid\rho\le d_2(H)\le\rho+n\rbrace,
\end{equation}

\noindent
since \(d_1(\mathfrak{G})=d_1(\mathfrak{G}/\mathfrak{G}^\prime)=d_1(\mathrm{Cl}_p(k))=\rho\).

\end{definition}

\begin{example}
\label{exm:ShafarevichCover}
Note first that the various \(n\)-covers of \(\mathfrak{G}\) form a non-descending chain of sets,
\[\mathrm{cov}_0(\mathfrak{G})\subseteq\mathrm{cov}_1(\mathfrak{G})\subseteq\ldots\subseteq\mathrm{cov}(\mathfrak{G}).\]
The cardinality of the cover can probably take any value.
Even for closely related groups, such as parent-descendant pairs, the values may be very different.
\begin{enumerate}

\item
Generally, if \(\#\mathrm{cov}(\mathfrak{G})=1\)
but the Shafarevich cover of \(\mathfrak{G}\) with respect to \(k\) is empty,
then \(\mathfrak{G}\) cannot be the second \(p\)-class group \(\mathrm{G}_p^2(k)\) of \(k\).

\item
For instance, the cover of the \(3\)-group \(\mathfrak{G}=\langle 243,4\rangle\) with
IPAD \(\tau^{(1)}=\lbrack 1^2;((1^3)^3,21)\rbrack\) and TKT \(\varkappa_1=(4111)\)
consists of \(\mathfrak{G}\) alone.
However, since \(d_2(\mathfrak{G})\ge\mathrm{MR}(\mathfrak{G})=3\) (\(p\)-multiplicator rank),
the Shafarevich cover, \(\mathrm{cov}_k(\mathfrak{G})=\mathrm{cov}_0(\mathfrak{G})\),
with respect to \textit{complex} quadratic fields \(k\), having this IPAD and TKT, is empty,
since \(r=0\), \(k\) with \(\rho=2\) does not contain a primitive third root of unity,
and thus \(n=0\).

\item
The \(3\)-group \(\mathfrak{G}=\langle 729,45\rangle\) is the unique immediate descendant of
\(\langle 243,4\rangle\) possessing a GI-automorphism.
It shares the common IPAD \(\tau^{(1)}=\lbrack 1^2;((1^3)^3,21)\rbrack\) and TKT \(\varkappa_1=(4111)\)
with its parent, and is infinitely capable with nuclear rank \(\mathrm{NR}(\mathfrak{G})=2\).
As mentioned in
\cite[Cor.6.2, p.301]{Ma6},
its cover \(\mathrm{cov}(\mathfrak{G})=\mathcal{T}(\mathfrak{G})\)
coincides with its complete infinite descendant tree,
and even the Shafarevich cover with respect to \textit{complex} quadratic fields \(k\),
\(\mathrm{cov}_k(\mathfrak{G})=\mathrm{cov}_0(\mathfrak{G})=\lbrace S_j\mid j\ge 0\rbrace\),
where \(S_j=\langle 243,4\rangle(-\#1;1-\#2;1)^j-\#1;1-\#2;2\) denotes a series of Schur \(\sigma\)-groups
starting with order \(3^8=6561\) for \(j=0\)
\cite[\S\ 6.2.2, pp.299--304]{Ma6},
is infinite, according to Bartholdi and Bush
\cite{BaBu}.
The Shafarevich cover with respect to \textit{real} quadratic fields \(K\),
\(\mathrm{cov}_K(\mathfrak{G})=\mathrm{cov}_1(\mathfrak{G})\), is more extensive,
containing the four groups \(\langle 2187,270\ldots 273\rangle\)
\cite[Fig.1, p.302]{Ma6}
as the smallest additional elements.
\end{enumerate}

\end{example}



\section{Constraints for higher \(p\)-class groups}
\label{s:Constraints}

Several constraints restrict the possibilities for higher \(p\)-class groups
\(\mathrm{G}_p^n(k)=\mathrm{Gal}(\mathrm{F}_p^n(k)\vert k)\), with \(n\ge 2\) or \(n=\infty\),
of a given class of algebraic number fields \(k\).
It is not always true that the most restrictive constraints
are most helpful for finding a certain group \(\mathrm{G}_p^n(k)\).
For instance, it is sometimes easier to determine the \(p\)-tower group \(\mathrm{G}_p^\infty(K)\)
of a \textit{real} quadratic field \(K\) with looser conditions, which admit lower group orders,
than of a \textit{complex} quadratic field \(k\) with very tight constraints,
which can only be satisfied by groups with huge orders.

Let us compile a \textit{list of constraints} proceeding from the most general to more particular cases:

\begin{enumerate}
\item
the Shafarevich bound for the relation rank of the \(p\)-tower group \(G:=\mathrm{G}_p^\infty(k)\),
\[d_1(G)\le d_2(G)\le d_1(G)+n,\]
with \(n=r+1\) if \(\zeta\in k\) and \(n=r\) if \(\zeta\notin k\),
\item
the cardinality \(\#\mathrm{cov}_k(\mathfrak{G})\) of the Shafarevich cover of the second \(p\)-class group
\(\mathfrak{G}:=\mathrm{G}_p^2(k)\),
\item
the existence of a \textit{generator inverting} (GI-)automorphism of any \(\mathrm{G}_p^n(K)\),
with \(n\ge 2\) or \(n=\infty\), acting as inversion map on the abelianization,
\item
\textit{selection rules} for the class \(\mathrm{cl}(\mathfrak{G})\),
due to \(p\)-class number relations for odd primes \(p\ge 3\),
\item
the \textit{capitulation number} \(\nu=\#\lbrace i\mid\varkappa_1(i)=0\rbrace\),
that is the number of total transfer kernels.
\end{enumerate}

\noindent
The first two items of this list concern \textit{arbitrary} number fields \(k\),
whereas the last three conditions are mainly motivated by (real or complex) \textit{quadratic} fields
\(K=\mathbb{Q}(\sqrt{d})\),
in particular, the very last item only by \textit{complex} quadratic fields with \(d<0\),
which enforce \(\nu=0\)
\cite{CgFt,Ma1}.



\section{Capitulation Type c.18, \((0122)\)}
\label{s:c18}

\subsection{The Ground State of Capitulation Type c.18}
\label{ss:c18GS}

Assume that \(K\) is an algebraic number field
with \(3\)-class group \(\mathrm{Cl}_3(K)\) of type \((3,3)\hat{=}1^2\).
Let \(L_1,\ldots,L_4\) be the four unramified cyclic cubic extensions of \(K\),
and suppose that the \(1^{\mathrm{st}}\) order Artin pattern \(\mathrm{AP}^{(1)}(K)\) of \(K\)
is given by the IPAD 
\(\tau^{(1)}(K)=\lbrack 1^2;\mathbf{2^2},1^3,21,21\rbrack\)
and the IPOD
\(\varkappa^{(1)}(K)=\lbrack G^\prime;\mathbf{G},H_1,H_2,H_2\rbrack\),
i.e. \(\tau_1(K)=((9,9),(3,3,3),(9,3),(9,3))\) are the type invariants of the \(3\)-class groups of the \(L_i\)
and \(\varkappa_1(K)=(0122)\) is the \(3\)-capitulation type of \(K\) in the \(L_i\).

\begin{proposition}
\label{prp:c18GS}

(D.C. Mayer, \(2010\))

\begin{enumerate}

\item
The second \(3\)-class group \(\mathfrak{G}=\mathrm{G}_3^2(K)\) of \(K\)
is isomorphic to the metabelian \(3\)-group \(\langle 729,49\rangle\)
and thus its relation rank satisfies the inequality
\(d_2(\mathfrak{G})\ge d_1(\mathfrak{G})+2\).

\item
Consequently,
if \(K\) is a number field with torsionfree unit rank \(r=1\)
and does not contain the third roots of unity
(in particular, if \(K=\mathbb{Q}(\sqrt{d})\) is a real quadratic field),
then the length of the \(3\)-class field tower of \(K\) must be \(\ell_3(K)\ge 3\).

\end{enumerate}

\end{proposition}

\begin{proof}

\begin{enumerate}

\item
First, we prove that \(\mathfrak{G}\) must have coclass \(\mathrm{cc}(\mathfrak{G})=2\).
This can be done in two ways,
either using \(\tau_1(K)\) alone or using \(\varkappa_1(K)\) alone.

\begin{itemize}
\item
According to items 1) and 3) of
\cite[Thm.3.2, p.291]{Ma6},
\(\tau_1(\mathfrak{G})=\tau_1(K)\)
must contain three components of type \(1^2\) if \(\mathrm{cc}(\mathfrak{G})=1\),
and two components of type \(1^3\) if \(\mathrm{cc}(\mathfrak{G})\ge 3\),
whence \(\mathrm{cc}(\mathfrak{G})=2\) is the only possibility for
\(\tau_1(\mathfrak{G})=(2^2,1^3,(21)^2)\) without any \(1^2\) and with only one \(1^3\).
\item
According to
\cite[Thm.2.5, p.479]{Ma2},
\(\varkappa_1(\mathfrak{G})=\varkappa_1(K)\) must contain three total kernels,
designated by \(0\), if \(\mathrm{cc}(\mathfrak{G})=1\).
Since all metabelian \(3\)-groups of coclass bigger than \(2\)
are descendants of \(\langle 3^5,3\rangle\),
and the TKT \((2100)\) of this root contains a \(2\)-cycle,
the TKT \(\varkappa_1(\mathfrak{G})\) must also contain a \(2\)-cycle,
if \(\mathrm{cc}(\mathfrak{G})\ge 3\),
according to
\cite[Cor.3.0.2, p.772]{BuMa}.
Therefore, the only possibility for \(\varkappa_1(\mathfrak{G})=(0122)\)
without a \(2\)-cycle and with only one \(0\)
is \(\mathrm{cc}(\mathfrak{G})=2\).
\end{itemize}

Next, we show that \(\mathfrak{G}\) must be a descendant of \(\langle 3^5,6\rangle\).
According to item 2) of
\cite[Thm.3.1, p.290]{Ma6},
the polarized component \(2^2\) of order \(3^4\) of \(\tau_1(\mathfrak{G})=\tau_1(K)\)
cannot occur for a sporadic \(3\)-group outside of coclass-\(2\) trees,
whence \(\mathfrak{G}\) must be a vertex of one of the three coclass-\(2\) trees
with metabelian mainline.
Their roots are
\(\langle 3^5,3\rangle\), resp. \(\langle 3^5,6\rangle\), resp. \(\langle 3^5,8\rangle\).
All descendants of these roots have three stable components of their TTT,
\(((1^3)^2,21)\), resp. \((1^3,(21)^2)\), resp. \(((21)^3)\),
whence \(\tau_1(\mathfrak{G})=(2^2,1^3,(21)^2)\)
unambiguously leads to a descendant of \(\langle 3^5,6\rangle\), by item 2) of
\cite[Thm.3.2, p.291]{Ma6}.
Furthermore, we have
\(d_2(\mathfrak{G})\ge\mathrm{MR}(\mathfrak{G})=4=d_1(\mathfrak{G})+2\),
since the \(p\)-multiplicator rank is a lower bound for the relation rank.

Finally, the Artin pattern
\(\mathrm{AP}(\mathfrak{G})=(\varkappa_1(\mathfrak{G}),\tau_1(\mathfrak{G}))\)
provides a sort of coordinate system
in which the coclass tree with root \(\langle 3^5,6\rangle\) is embedded,
with horizontal axis \(\varkappa_1(\mathfrak{G})\) and vertical axis \(\tau_1(\mathfrak{G})\).
The polarization \(2^2\) of \(\tau_1(\mathfrak{G})=(2^2,1^3,(21)^2)\)
determines the nilpotency class \(\mathrm{cl}(\mathfrak{G})=4\)
and thus also the order \(\lvert\mathfrak{G}\rvert=3^6\), according to item 2) of
\cite[Thm.3.2, p.291]{Ma6},
since the defect \(k=1\) is only possible for type H.4 \((2122)\).
The polarization \(0\) of \(\varkappa_1(\mathfrak{G})=(0122)\) with stable components \((122)\)
unambiguously identifies the mainline vertex of order \(3^6\),
which is \(\langle 3^6,49\rangle\), according to Figure
\ref{fig:TreeOverviewQ}.

\item
According to
\cite[Thm.6, p.140]{Sh},
the relation rank \(d_2(G)\) of the \(3\)-class tower group
\(G:=\mathrm{Gal}(\mathrm{F}_3^\infty(K)\vert K)\)
of a number field \(K\) with unit rank \(r=r_1+r_2-1=1\) and \(\zeta_3\notin K\)
must satisfy \(d_2(G)\le d_1(G)+r=2+1=3\),
whence \(G\not\simeq\mathfrak{G}\) and \(\ell_3(K)=\mathrm{dl}(G)\ge 3\).
\end{enumerate}
\end{proof}

\noindent
Suppose now that, under the assumptions preceding Proposition
\ref{prp:c18GS},
\(G=\mathrm{G}_3^\infty(K)\) denotes the \(3\)-class tower group of \(K\) and
we are additionally given the \(2^{\mathrm{nd}}\) order IPAD of \(K\),
\[\tau^{(2)}(K)=
\lbrack 1^2;(\mathbf{2^2};\tau_1(L_1)),(1^3;\tau_1(L_2)),(21;\tau_1(L_3)),(21;\tau_1(L_4))\rbrack,\]
with fixed \(\tau_1(L_1)=((21^2)^4)\) and  \(\tau_1(L_3)=(21^2,(21)^3)\).



\begin{theorem}
\label{thm:c18GS}

(D.C. Mayer, Aug. \(2015\))\\
Among the non-metabelian candidates for the \(3\)-tower group \(G\) of \(K\),
the immediate descendants of step size \(1\) of \(\langle 729,49\rangle\)
are distinguished by the following criteria:

\begin{enumerate}

\item
\(\tau_1(L_2)=(21^2,\mathbf{(21^2)^3},(1^2)^9)\), \(\tau_1(L_4)=(21^2,\mathbf{(21)^3})\)
\(\Longleftrightarrow\)
\(G\simeq\langle 2187,\mathbf{284}\rangle\),

\item
\(\tau_1(L_2)=(21^2,\mathbf{(1^3)^3},(1^2)^9)\), \(\tau_1(L_4)=(21^2,\mathbf{(31)^3})\)
\(\Longleftrightarrow\)
\(G\simeq\langle 2187,\mathbf{291}\rangle\).

\end{enumerate}

\noindent
In both cases, we have derived length \(\mathrm{dl}(G)=3\) and nilpotency class \(\mathrm{cl}(G)=5\).

\end{theorem}

\begin{proof}
We use the \(p\)-group generation algorithm by Newman
\cite{Nm}
and O'Brien
\cite{Ob},
which is implemented in our licence of the computational algebra system MAGMA
\cite{BCP,BCFS,MAGMA},
to construct the descendant tree \(\mathcal{T}(R)\) of the root \(R=\langle 243,6\rangle\),
which is restricted to the coclass tree \(\mathcal{T}^2(R)\) in Figure
\ref{fig:TreeOverviewQ}
by ignoring the bifurcation at the not coclass-settled vertex \(\langle 729,49\rangle\).
In Figure
\ref{fig:C18PrunedTreeQ},
however, all periodic bifurcations
\cite[\S\ 21.2]{Ma5}
in the complete descendant tree are taken into consideration,
but the tree is pruned from all TKTs different from c.18.
In parallel computation with the recursive tree construction,
the iterated IPAD of second order \(\tau^{(2)}(V)\) is determined for each vertex \(V\) and
non-metabelian vertices \(H\) are checked for their second derived quotient \(H/H^{\prime\prime}\).
The construction can be terminated at order \(3^{11}\), because
several components of the \(2^{\mathrm{nd}}\) order IPAD become stable
and the remaining components reveal a deterministic growth: we have
\[\tau_1(L_2)=(\ast,\mathbf{(1^3)^3},(1^2)^9),\ \tau_1(L_3)=\tau_1(L_4)=(\ast,\mathbf{(21)^3})\text{ for vertices of coclass }\mathbf{2}\]
and
\[\tau_1(L_2)=(\ast,\mathbf{(21^2)^3},(1^2)^9),\ \tau_1(L_3)=\tau_1(L_4)=(\ast,\mathbf{(31)^3})\text{ for vertices of coclass }\mathbf{3}.\]
Therefore, the vertices \(\langle 2187,284\rangle\) and \(\langle 2187,291\rangle\)
are characterized uniquely by their iterated IPAD of second order,
and the cover of their common parent is given by

\(\mathrm{cov}(\langle 729,49\rangle)=\lbrace\langle 729,49\rangle,\langle 2187,284\rangle,\langle 2187,291\rangle\rbrace\).
\end{proof}



\subsection{Real Quadratic Fields of Type c.18}
\label{ss:RQFc18GS}

\begin{proposition}
\label{prp:RQFc18GS}

(D.C. Mayer, Feb. \(2010\))\\
In the range \(0<d<10^7\) of fundamental discriminants \(d\)
of real quadratic fields \(K=\mathbb{Q}(\sqrt{d})\),
there exist precisely \(\mathbf{28}\) cases
with \(3\)-capitulation type \(\varkappa_1(K)=(0122)\).

\end{proposition}

\begin{proof}
The results
\cite[Tbl.6.5, p.452]{Ma3},
where the entry in column freq. should be \(28\) instead of \(29\) in the first row
and \(4\) instead of \(3\) in the fourth row,
were computed by means of the free number theoretic computer algebra system PARI/GP
\cite{PARI}
using an implementation of our own principalization algorithm in a PARI script, as described in detail in
\cite[\S\ 5, pp.446--450]{Ma3}.
\end{proof}

\begin{theorem} 
\label{thm:RQFc18GS}

(D.C. Mayer, Aug. \(2015\))

\begin{enumerate}

\item
The \(\mathbf{10}\) real quadratic fields \(K=\mathbb{Q}(\sqrt{d})\) with the following discriminants \(d\)
(\(\mathbf{36}\%\) of \(28\)),
\[
\begin{aligned}
1\,030\,117 &,& 3\,259\,597 &,& 3\,928\,632 &,& 4\,593\,673 &,& 5\,327\,080, \\
5\,909\,813 &,& 7\,102\,277 &,& 7\,738\,629 &,& 7\,758\,589 &,& 9\,583\,736, 
\end{aligned}
\]

\noindent
have \(3\)-class tower group \(G\simeq\langle 3^7,\mathbf{284}\rangle\) and \(3\)-tower length \(\ell_3(K)=3\).

\item
The \(\mathbf{18}\) real quadratic fields \(K=\mathbb{Q}(\sqrt{d})\) with the following discriminants \(d\)
(\(\mathbf{64}\%\) of \(28\)),
\[
\begin{aligned}
   534\,824 &,& 2\,661\,365 &,& 2\,733\,965 &,& 3\,194\,013 &,& 3\,268\,781, \\
4\,006\,033 &,& 5\,180\,081 &,& 5\,250\,941 &,& 5\,489\,661 &,& 6\,115\,852, \\
6\,290\,549 &,& 7\,712\,184 &,& 7\,857\,048 &,& 7\,943\,761 &,& 8\,243\,113, \\
8\,747\,997 &,& 8\,899\,661 &,& 9\,907\,837 &,&  &&  
\end{aligned}
\]

\noindent
have \(3\)-class tower group \(G\simeq\langle 3^7,\mathbf{291}\rangle\) and \(3\)-tower length \(\ell_3(K)=3\).

\end{enumerate}

\end{theorem}

\begin{proof}
Since all these real quadratic fields \(K=\mathbb{Q}(\sqrt{d})\) have
\(3\)-capitulation type \(\varkappa_1(K)=(0122)\) and \(1^{\mathrm{st}}\) IPAD  
\(\tau^{(1)}(K)=\lbrack 1^2;\mathbf{2^2},1^3,(21)^2\rbrack\),
and the \(10\) fields in the first list have \(2^{\mathrm{nd}}\) IPAD
\[\tau_1(L_1)=((21^2)^4),\ \tau_1(L_2)=(21^2,\mathbf{(21^2)^3},(1^2)^9),\ \tau_1(L_3)=(21^2,(21)^3),\ \tau_1(L_4)=(21^2,\mathbf{(21)^3}),\]
whereas the \(18\) fields in the second list have \(2^{\mathrm{nd}}\) IPAD
\[\tau_1(L_1)=((21^2)^4),\ \tau_1(L_2)=(21^2,\mathbf{(1^3)^3},(1^2)^9),\ \tau_1(L_3)=(21^2,(21)^3),\ \tau_1(L_4)=(21^2,\mathbf{(31)^3}),\]
the claim is a consequence of Theorem
\ref{thm:c18GS}.
\end{proof}

\begin{remark}
\label{rmk:RQFc18GS}
Unpublished results of M.R. Bush show that there are \(4318\) real quadratic fields
with discriminants \(0<d<10^9\) having the IPAD in Proposition
\ref{prp:RQFc18GS}.
For the remaining \(4290\) cases outside of the range \(0<d<10^7\),
we cannot specify the \(3\)-tower group,
since the computations would require exceeding amounts of CPU time.
However, according to Proposition
\ref{prp:c18GS}
and Theorem
\ref{thm:c18GS},
we know for sure that the length of the \(3\)-tower is certainly \(\ell_3(K)=3\). 
\end{remark}



Figure
\ref{fig:C18PrunedTreeQ}
visualizes the groups in Theorems
\ref{thm:c18GS}
and
\ref{thm:c18ES1}
and their population in Theorems
\ref{thm:RQFc18GS}
and
\ref{thm:RQFc18ES1}.

{\tiny

\begin{figure}[hb]
\caption{Non-metabelian \(3\)-tower groups \(G\) on the pruned tree \(\mathcal{T}_\ast(\langle 243,6\rangle)\)}
\label{fig:C18PrunedTreeQ}


\setlength{\unitlength}{0.8cm}
\begin{picture}(18,22)(-6,-21)

\put(-5,0.5){\makebox(0,0)[cb]{Order}}
\put(-5,0){\line(0,-1){18}}
\multiput(-5.1,0)(0,-2){10}{\line(1,0){0.2}}
\put(-5.2,0){\makebox(0,0)[rc]{\(243\)}}
\put(-4.8,0){\makebox(0,0)[lc]{\(3^5\)}}
\put(-5.2,-2){\makebox(0,0)[rc]{\(729\)}}
\put(-4.8,-2){\makebox(0,0)[lc]{\(3^6\)}}
\put(-5.2,-4){\makebox(0,0)[rc]{\(2\,187\)}}
\put(-4.8,-4){\makebox(0,0)[lc]{\(3^7\)}}
\put(-5.2,-6){\makebox(0,0)[rc]{\(6\,561\)}}
\put(-4.8,-6){\makebox(0,0)[lc]{\(3^8\)}}
\put(-5.2,-8){\makebox(0,0)[rc]{\(19\,683\)}}
\put(-4.8,-8){\makebox(0,0)[lc]{\(3^9\)}}
\put(-5.2,-10){\makebox(0,0)[rc]{\(59\,049\)}}
\put(-4.8,-10){\makebox(0,0)[lc]{\(3^{10}\)}}
\put(-5.2,-12){\makebox(0,0)[rc]{\(177\,147\)}}
\put(-4.8,-12){\makebox(0,0)[lc]{\(3^{11}\)}}
\put(-5.2,-14){\makebox(0,0)[rc]{\(531\,441\)}}
\put(-4.8,-14){\makebox(0,0)[lc]{\(3^{12}\)}}
\put(-5.2,-16){\makebox(0,0)[rc]{\(1\,594\,323\)}}
\put(-4.8,-16){\makebox(0,0)[lc]{\(3^{13}\)}}
\put(-5.2,-18){\makebox(0,0)[rc]{\(4\,782\,969\)}}
\put(-4.8,-18){\makebox(0,0)[lc]{\(3^{14}\)}}
\put(-5,-18){\vector(0,-1){2}}

\put(0.1,0.2){\makebox(0,0)[lb]{\(\langle 6\rangle\)}}
\put(0.1,-1.8){\makebox(0,0)[lb]{\(\langle 49\rangle\) (not coclass-settled)}}
\put(1.1,-2.8){\makebox(0,0)[lb]{\(1^{\text{st}}\) bifurcation}}
\put(0.1,-3.8){\makebox(0,0)[lb]{\(\langle 285\rangle\)}}
\put(0.1,-5.8){\makebox(0,0)[lb]{\(1;1\)}}
\put(0.1,-7.8){\makebox(0,0)[lb]{\(1;1\)}}
\put(0.1,-9.8){\makebox(0,0)[lb]{\(1;1\)}}
\put(0.1,-11.8){\makebox(0,0)[lb]{\(1;1\)}}
\multiput(0,0)(0,-2){7}{\circle*{0.2}}
\multiput(0,0)(0,-2){6}{\line(0,-1){2}}
\put(0,-12){\vector(0,-1){2}}
\put(-0.2,-14.2){\makebox(0,0)[rt]{\(\mathcal{T}_\ast^2(\langle 243,6\rangle)\)}}

\put(-2,-4.2){\makebox(0,0)[ct]{\(\langle 291\rangle\)}}
\put(-2,-8.2){\makebox(0,0)[ct]{\(1;7\)}}
\put(-2,-12.2){\makebox(0,0)[ct]{\(1;7\)}}
\multiput(0,-2)(0,-4){3}{\line(-1,-1){2}}
\multiput(-2.05,-4.05)(0,-4){3}{\framebox(0.1,0.1){}}

\put(-1,-4.2){\makebox(0,0)[ct]{\(\langle 284\rangle\)}}
\multiput(0,-2)(0,-4){1}{\line(-1,-2){1}}
\multiput(-1.05,-4.05)(0,-4){1}{\framebox(0.1,0.1){}}

\put(0,-2){\line(1,-1){4}}

\put(4.1,-5.8){\makebox(0,0)[lb]{\(2;1\)}}
\put(4.1,-7.8){\makebox(0,0)[lb]{\(1;1\) (not coclass-settled)}}
\put(5.1,-8.8){\makebox(0,0)[lb]{\(2^{\text{nd}}\) bifurcation}}
\put(4.1,-9.8){\makebox(0,0)[lb]{\(1;2\)}}
\put(4.1,-11.8){\makebox(0,0)[lb]{\(1;1\)}}
\put(4.1,-13.8){\makebox(0,0)[lb]{\(1;1\)}}
\multiput(3.95,-6.05)(0,-2){5}{\framebox(0.1,0.1){}}
\multiput(4,-6)(0,-2){4}{\line(0,-1){2}}
\put(4,-14){\vector(0,-1){2}}
\put(3.8,-16.2){\makebox(0,0)[rt]{\(\mathcal{T}_\ast^3(\langle 729,49\rangle-\#2;1)\)}}

\put(2,-10.2){\makebox(0,0)[ct]{\(1;8\)}}
\put(2,-14.2){\makebox(0,0)[ct]{\(1;7\)}}
\multiput(4,-8)(0,-4){2}{\line(-1,-1){2}}
\multiput(1.95,-10.05)(0,-4){2}{\framebox(0.1,0.1){}}

\put(3,-10.2){\makebox(0,0)[ct]{\(1;1\)}}
\multiput(4,-8)(0,-4){1}{\line(-1,-2){1}}
\multiput(2.95,-10.05)(0,-4){1}{\framebox(0.1,0.1){}}

\put(4,-8){\line(1,-1){4}}

\put(8.1,-11.8){\makebox(0,0)[lb]{\(2;1\)}}
\put(8.1,-13.8){\makebox(0,0)[lb]{\(1;1\) (not coclass-settled)}}
\put(9.1,-14.8){\makebox(0,0)[lb]{\(3^{\text{rd}}\) bifurcation}}
\put(8.1,-15.8){\makebox(0,0)[lb]{\(1;2\)}}
\multiput(7.95,-12.05)(0,-2){3}{\framebox(0.1,0.1){}}
\multiput(8,-12)(0,-2){2}{\line(0,-1){2}}
\put(8,-16){\vector(0,-1){2}}
\put(7.8,-18.2){\makebox(0,0)[rt]{\(\mathcal{T}_\ast^4(\langle 729,49\rangle-\#2;1-\#1;1-\#2;1)\)}}

\put(6,-16.2){\makebox(0,0)[ct]{\(1;8\)}}
\multiput(8,-14)(0,-4){1}{\line(-1,-1){2}}
\multiput(5.95,-16.05)(0,-4){1}{\framebox(0.1,0.1){}}

\put(7,-16.2){\makebox(0,0)[ct]{\(1;1\)}}
\multiput(8,-14)(0,-4){1}{\line(-1,-2){1}}
\multiput(6.95,-16.05)(0,-4){1}{\framebox(0.1,0.1){}}

\put(8,-14){\line(1,-1){4}}

\put(12.1,-17.8){\makebox(0,0)[lb]{\(2;1\)}}
\multiput(11.95,-18.05)(0,-2){1}{\framebox(0.1,0.1){}}
\put(12,-18){\vector(0,-1){2}}
\put(11.8,-19.9){\makebox(0,0)[rt]{\(\mathcal{T}_\ast^5(\langle 729,49\rangle-\#2;1-\#1;1-\#2;1-\#1;1-\#2;1)\)}}



\multiput(0.3,-2.2)(0,-4){3}{\oval(1,1.5)}
\put(-0.6,-2.1){\makebox(0,0)[rc]{\underbar{\textbf{\(\#4318\)}}}}
\put(-0.6,-6.1){\makebox(0,0)[rc]{\underbar{\textbf{\(\#138\)}}}}
\put(-0.6,-10.1){\makebox(0,0)[rc]{\underbar{\textbf{\(\#5\)}}}}

\multiput(-2,-4.25)(0,-4){3}{\oval(1,1.5)}
\put(-2,-5.3){\makebox(0,0)[rc]{\underbar{\textbf{\(+534\,824\)}}}}
\put(-2,-9.3){\makebox(0,0)[cc]{\underbar{\textbf{\(+13\,714\,789\)}}}}
\put(-2,-13.3){\makebox(0,0)[cc]{\underbar{\textbf{\(+241\,798\,776\)}}}}

\multiput(-1,-4.25)(0,-4){1}{\oval(1,1.5)}
\put(-1,-5.3){\makebox(0,0)[lc]{\underbar{\textbf{\(+1\,030\,117\)}}}}

\multiput(2.5,-10.25)(4,-6){2}{\oval(2,1.5)}
\put(2.5,-11.3){\makebox(0,0)[cc]{\underbar{\textbf{\(+14\,252\,156\)}}}}

\put(3,-15){\vector(-1,1){0.5}}
\put(4,-15.3){\makebox(0,0)[cc]{\underbar{\textbf{\(+174\,458\,681\)}}}}
\put(5,-15.5){\vector(1,-1){0.5}}

\multiput(2,-14.25)(4,-6){1}{\oval(1,1.5)}

\end{picture}

\end{figure}

}



\subsection{The 1\({}^{\text{st}}\) Excited State of Capitulation Type c.18}
\label{ss:c18ES1}

As before, \(K\) is an algebraic number field
with \(3\)-class group \(\mathrm{Cl}_3(K)\) of type \((3,3)\hat{=}1^2\).
Let \(L_1,\ldots,L_4\) be the four unramified cyclic cubic extensions of \(K\),
and suppose that the \(1^{\mathrm{st}}\) order Artin pattern \(\mathrm{AP}^{(1)}(K)\) of \(K\)
is given by the IPAD 
\(\tau^{(1)}(K)=\lbrack 1^2;\mathbf{3^2},1^3,21,21\rbrack\)
and the IPOD
\(\varkappa^{(1)}(K)=\lbrack G^\prime;\mathbf{G},H_1,H_2,H_2\rbrack\),
i.e. \(\tau_1(K)=((27,27),(3,3,3),(9,3),(9,3))\) are the type invariants of the \(3\)-class groups of the \(L_i\)
and \(\varkappa_1(K)=(0122)\) is the \(3\)-capitulation type of \(K\) in the \(L_i\).

\begin{proposition}
\label{prp:c18ES1}

(D.C. Mayer, \(2010\))

\begin{enumerate}

\item
The second \(3\)-class group \(\mathfrak{G}=\mathrm{G}_3^2(K)\) of \(K\)
is isomorphic to the metabelian \(3\)-group \(\langle 2187,285\rangle-\#1;1\)
and thus its relation rank satisfies the inequality
\(d_2(\mathfrak{G})\ge d_1(\mathfrak{G})+2\).

\item
Consequently,
if \(K\) is a number field with torsionfree unit rank \(r=1\)
and does not contain the third roots of unity
(in particular, if \(K=\mathbb{Q}(\sqrt{d})\) is a real quadratic field),
then the length of the \(3\)-class field tower of \(K\) must be \(\ell_3(K)\ge 3\).

\end{enumerate}

\end{proposition}

\begin{proof}

A great deal of the proof is similar to the proof of Proposition
\ref{prp:c18GS}.

\begin{enumerate}

\item
First, we prove that \(\mathfrak{G}\) must have coclass \(\mathrm{cc}(\mathfrak{G})=2\).
Next, we show that \(\mathfrak{G}\) must be a descendant of \(\langle 3^5,6\rangle\),
this time using the polarized component \(3^2\) of order \(3^6\) of \(\tau_1(\mathfrak{G})\).
Again, we have
\(d_2(\mathfrak{G})\ge\mathrm{MR}(\mathfrak{G})=4=d_1(\mathfrak{G})+2\),
since the \(p\)-multiplicator rank is a lower bound for the relation rank.
Finally, the polarization \(3^2\) of \(\tau_1(\mathfrak{G})=(3^2,1^3,(21)^2)\)
determines the nilpotency class \(\mathrm{cl}(\mathfrak{G})=6\)
and thus also the order \(\lvert\mathfrak{G}\rvert=3^8\).
The polarization \(0\) of \(\varkappa_1(\mathfrak{G})=(0122)\) with stable components \((122)\)
unambiguously identifies the mainline vertex of order \(3^8\),
which is \(\langle 2187,285\rangle-\#1;1\), according to Figure
\ref{fig:C18PrunedTreeQ}.

\item
As before,
\cite[Thm.6, p.140]{Sh}
implies that the relation rank \(d_2(G)\) of the \(3\)-class tower group
\(G:=\mathrm{Gal}(\mathrm{F}_3^\infty(K)\vert K)\)
of a number field \(K\) with unit rank \(r=r_1+r_2-1=1\) and \(\zeta_3\notin K\)
satisfies \(d_2(G)\le d_1(G)+r=2+1=3\),
whence \(G\not\simeq\mathfrak{G}\) and \(\ell_3(K)=\mathrm{dl}(G)\ge 3\).
\end{enumerate}
\end{proof}

\noindent
Suppose now that, under the assumptions preceding Proposition
\ref{prp:c18ES1},
\(G=\mathrm{G}_3^\infty(K)\) denotes the \(3\)-class tower group of \(K\) and
we are additionally given the \(2^{\mathrm{nd}}\) order IPAD of \(K\),
\[\tau^{(2)}(K)=
\lbrack 1^2;(\mathbf{3^2};\tau_1(L_1)),(1^3;\tau_1(L_2)),(21;\tau_1(L_3)),(21;\tau_1(L_4))\rbrack,\]
with fixed \(\tau_1(L_1)=((321)^4)\).

\begin{theorem}
\label{thm:c18ES1}

(D.C. Mayer, Aug. \(2015\))

\begin{enumerate}

\item
\(\tau_1(L_2)=(321,\mathbf{(1^3)^3},(1^2)^9)\) and \(\tau_1(L_i)=(321,\mathbf{(21)^3)}\) for \(i=3,4\)
\(\Longleftrightarrow\)\\
\(G\simeq\langle 3^7,285\rangle-\mathbf{\#1;1}\), of order \(\lvert G\rvert=3^8\),
or\\
\(G\simeq\langle 3^7,285\rangle-\mathbf{\#1;1-\#1;7}\), of order \(\lvert G\rvert=3^9\),

\item
\(\tau_1(L_2)=(321,\mathbf{(21^2)^3},(1^2)^9)\) and \(\tau_1(L_i)=(321,\mathbf{(31)^3})\) for \(i=3,4\)
\(\Longleftrightarrow\)\\
\(G\simeq\langle 3^6,49\rangle-\mathbf{\#2;1-\#1;1}\), of order \(\lvert G\rvert=3^9\),
or\\
\(G\simeq\langle 3^6,49\rangle-\mathbf{\#2;1-\#1;1-\#1;1}\), of order \(\lvert G\rvert=3^{10}\),
or\\
\(G\simeq\langle 3^6,49\rangle-\mathbf{\#2;1-\#1;1-\#1;8}\), of order \(\lvert G\rvert=3^{10}\).

\end{enumerate}

\end{theorem}

\begin{proof}
Similar as in the proof of Theorem
\ref{thm:c18GS},
we construct the descendant tree \(\mathcal{T}(R)\) of the root \(R=\langle 243,6\rangle\),
determine the iterated IPAD of second order \(\tau^{(2)}(V)\) for each vertex \(V\),
and check the second derived quotient \(H/H^{\prime\prime}\) of non-metabelian vertices \(H\).

It turns out that the cover of the common metabelianization
\(\mathfrak{G}=\langle 3^7,285\rangle-\#1;1=G/G^{\prime\prime}\)
of all candidates \(G\) for the \(3\)-tower group is given by
\(\mathrm{cov}(\langle 3^7,285\rangle-\#1;1)=\)\\
\(\lbrace\langle 3^7,285\rangle-\#1;1,\ \langle 3^7,285\rangle-\#1;1-\#1;7,\)\\
\(\langle 3^6,49\rangle-\#2;1-\#1;1,\ \langle 3^6,49\rangle-\#2;1-\#1;1-\#1;1,\ \langle 3^6,49\rangle-\#2;1-\#1;1-\#1;8\rbrace\).

The various vertices 
are not characterized uniquely by their iterated IPAD of second order.
Rather they can be identified as batches of two resp. three vertices
in the claimed manner.
\end{proof}

\begin{corollary} 
\label{cor:c18ES1}

(D.C. Mayer, Aug. \(2015\))\\
If \(K\) has torsionfree unit rank \(1\) and does not contain a primitive third root of unity, then

\begin{enumerate}

\item
\(\tau_1(L_2)=(321,\mathbf{(1^3)^3},(1^2)^9)\) and \(\tau_1(L_i)=(321,\mathbf{(21)^3)}\) for \(i=3,4\)
\(\Longleftrightarrow\)\\
\(G\simeq\langle 3^7,285\rangle-\mathbf{\#1;1-\#1;7}\).

\item
\(\tau_1(L_2)=(321,\mathbf{(21^2)^3},(1^2)^9)\) and \(\tau_1(L_i)=(321,\mathbf{(31)^3})\) for \(i=3,4\)
\(\Longleftrightarrow\)\\
\(G\simeq\langle 3^6,49\rangle-\mathbf{\#2;1-\#1;1-\#1;1}\) or\\
\(G\simeq\langle 3^6,49\rangle-\mathbf{\#2;1-\#1;1-\#1;8}\).

\end{enumerate}

\noindent
All these groups have derived length \(\mathrm{dl}(G)=3\).

\end{corollary}

\begin{proof}
The two infinitely capable groups
\(\langle 3^7,285\rangle-\mathbf{\#1;1}\)
and
\(\langle 3^6,49\rangle-\mathbf{\#2;1-\#1;1}\)
have \(p\)-multiplicator rank \(\mathrm{MR}(G)=4\) and
thus relation rank \(d_2(G)\ge 4\),
and consequently cannot satisfy the Shafarevich inequality \(d_2(G)\le d_1(G)+r=2+1=3\)
\cite[Thm.6, p.140]{Sh}
for a field \(K\) with unit rank \(r=r_1+r_2-1=1\) and \(\zeta_3\notin K\).
\end{proof}



\subsection{Real Quadratic Fields of Type c.18\(\uparrow\)}
\label{ss:RQFc18ES1}

\begin{proposition}
\label{prp:RQFc18ES1}

(M.R. Bush, Jul. \(2015\))\\
In the range \(0<d<10^8\) of fundamental discriminants \(d\)
of real quadratic fields \(K=\mathbb{Q}(\sqrt{d})\)
there exist precisely \(\mathbf{8}\) cases with \(1^{\mathrm{st}}\) IPAD
\(\tau^{(1)}(K)=\lbrack 1^2;\mathbf{3^2},1^3,(21)^2\rbrack\).

\end{proposition}

\begin{proof}
The results were communicated to us on July 11, 2015, by M.R. Bush, who used PARI/GP
\cite{PARI}
with similar techniques as described in our paper
\cite[\S\ 5, pp.446--450]{Ma3},
and additionally double-checked with MAGMA
\cite{MAGMA}.
\end{proof}

\begin{corollary}
\label{cor:RQFc18ES1}

(D.C. Mayer, \(2010\))\\
A quadratic field \(K=\mathbb{Q}(\sqrt{d})\)
with
\(\tau^{(1)}(K)=\lbrack 1^2;\mathbf{3^2},1^3,(21)^2\rbrack\)
must be a real quadratic field
with \(3\)-capitulation type \(\varkappa_1(K)=(0122)\).

\end{corollary}

\begin{proof}
This is a consequence of item (1) in
\cite[Cor.4.4.3, p.442]{Ma3}.
\end{proof}

\begin{theorem} 
\label{thm:RQFc18ES1}

(D.C. Mayer, Aug. \(2015\))

\begin{enumerate}

\item
The \(\mathbf{4}\) real quadratic fields \(K=\mathbb{Q}(\sqrt{d})\) with the following discriminants \(d\) (\(\mathbf{50}\%\) of 8),

\begin{center}
\(13\,714\,789,\ 24\,037\,912,\ 54\,683\,977,\ 94\,272\,565,\)  
\end{center}

\noindent
have \(3\)-class tower group
\(G\simeq\langle 3^7,285\rangle-\mathbf{\#1;1-\#1;7}\).

\item
The \(\mathbf{4}\) real quadratic fields \(K=\mathbb{Q}(\sqrt{d})\) with the following discriminants \(d\) (\(\mathbf{50}\%\) of 8),

\begin{center}
\(14\,252\,156,\ 46\,748\,181,\ 67\,209\,369,\ 78\,200\,897,\)
\end{center}

\noindent
have \(3\)-class tower group either\\
\(G\simeq\langle 3^6,49\rangle-\mathbf{\#2;1-\#1;1-\#1;1}\) or\\
\(G\simeq\langle 3^6,49\rangle-\mathbf{\#2;1-\#1;1-\#1;8}\).

\end{enumerate}

\noindent
In each case, the length of the \(3\)-class tower of \(K\) is given by \(\ell_3(K)=3\).

\end{theorem}

\begin{proof}
Since all these real quadratic fields \(K=\mathbb{Q}(\sqrt{d})\) have
\(3\)-capitulation type \(\varkappa_1(K)=(0122)\), \(1^{\mathrm{st}}\) IPAD  
\(\tau^{(1)}(K)=\lbrack 1^2;\mathbf{3^2},1^3,(21)^2\rbrack\)
and suitable \(2^{\mathrm{nd}}\) IPAD,
the claim is a consequence of Corollary
\ref{cor:c18ES1}.
\end{proof}

\begin{remark}
\label{rmk:RQFc18ES1}
Of course, the percentages given in Theorem
\ref{thm:RQFc18ES1}
are unable to predict reliable tendencies for extensive statistical ensembles.
The results of M.R. Bush show that there are \(138\) real quadratic fields
with discriminants \(0<d<10^9\) having the IPAD in Proposition
\ref{prp:RQFc18ES1}.
For the remaining \(130\) cases outside of the range \(0<d<10^8\),
we cannot specify the \(3\)-tower group,
since the computation would require too much CPU time.
However, according to Proposition
\ref{prp:c18ES1}
and Corollary
\ref{cor:c18ES1},
we can be sure that the length of the \(3\)-tower is exactly \(\ell_3(K)=3\). 
\end{remark}



\subsection{The 2\({}^{\text{nd}}\) Excited State of Capitulation Type c.18}
\label{ss:c18ES2}

As before, \(K\) is an algebraic number field
with \(3\)-class group \(\mathrm{Cl}_3(K)\) of type \((3,3)\hat{=}1^2\).
Let \(L_1,\ldots,L_4\) be the four unramified cyclic cubic extensions of \(K\),
and suppose that the \(1^{\mathrm{st}}\) order Artin pattern \(\mathrm{AP}^{(1)}(K)\) of \(K\)
is given by the IPAD 
\(\tau^{(1)}(K)=\lbrack 1^2;\mathbf{4^2},1^3,21,21\rbrack\)
and the IPOD
\(\varkappa^{(1)}(K)=\lbrack G^\prime;\mathbf{G},H_1,H_2,H_2\rbrack\),
i.e. \(\tau_1(K)=((81,81),(3,3,3),(9,3),(9,3))\) are the type invariants of the \(3\)-class groups of the \(L_i\)
and \(\varkappa_1(K)=(0122)\) is the \(3\)-capitulation type of \(K\) in the \(L_i\).

\begin{proposition}
\label{prp:c18ES2}

(D.C. Mayer, \(2010\))

\begin{enumerate}

\item
The second \(3\)-class group \(\mathfrak{G}=\mathrm{G}_3^2(K)\) of \(K\)
is isomorphic to the metabelian \(3\)-group \(\langle 2187,285\rangle(-\#1;1)^3\)
and thus its relation rank satisfies the inequality
\(d_2(\mathfrak{G})\ge d_1(\mathfrak{G})+2\).

\item
Consequently,
if \(K\) is a number field with torsionfree unit rank \(r=1\)
and does not contain the third roots of unity
(in particular, if \(K=\mathbb{Q}(\sqrt{d})\) is a real quadratic field),
then the length of the \(3\)-class field tower of \(K\) must be \(\ell_3(K)\ge 3\).

\end{enumerate}

\end{proposition}

\begin{proof}

Again, a great deal of the proof is similar to the proof of Proposition
\ref{prp:c18GS}.

\begin{enumerate}

\item
First, we prove that \(\mathfrak{G}\) must have coclass \(\mathrm{cc}(\mathfrak{G})=2\).
Next, we show that \(\mathfrak{G}\) must be a descendant of \(\langle 3^5,6\rangle\),
this time using the polarized component \(4^2\) of order \(3^6\) of \(\tau_1(\mathfrak{G})\).
Again, we have
\(d_2(\mathfrak{G})\ge\mathrm{MR}(\mathfrak{G})=4=d_1(\mathfrak{G})+2\),
since the \(p\)-multiplicator rank is a lower bound for the relation rank.
Finally, the polarization \(4^2\) of \(\tau_1(\mathfrak{G})=(3^2,1^3,(21)^2)\)
determines the nilpotency class \(\mathrm{cl}(\mathfrak{G})=8\)
and thus also the order \(\lvert\mathfrak{G}\rvert=3^{10}\).
The polarization \(0\) of \(\varkappa_1(\mathfrak{G})=(0122)\) with stable components \((122)\)
unambiguously identifies the mainline vertex of order \(3^{10}\),
which is \(\langle 2187,285\rangle(-\#1;1)^3\), according to Figure
\ref{fig:C18PrunedTreeQ}.

\item
As before,
\cite[Thm.6, p.140]{Sh}
implies that the relation rank \(d_2(G)\) of the \(3\)-class tower group
\(G:=\mathrm{Gal}(\mathrm{F}_3^\infty(K)\vert K)\)
of a number field \(K\) with unit rank \(r=r_1+r_2-1=1\) and \(\zeta_3\notin K\)
satisfies \(d_2(G)\le d_1(G)+r=2+1=3\),
whence \(G\not\simeq\mathfrak{G}\) and \(\ell_3(K)=\mathrm{dl}(G)\ge 3\).
\end{enumerate}
\end{proof}

\noindent
Suppose now that, under the assumptions preceding Proposition
\ref{prp:c18ES2},
\(G=\mathrm{G}_3^\infty(K)\) denotes the \(3\)-class tower group of \(K\) and
we are additionally given the \(2^{\mathrm{nd}}\) order IPAD of \(K\),
\[\tau^{(2)}(K)=
\lbrack 1^2;(\mathbf{4^2};\tau_1(L_1)),(1^3;\tau_1(L_2)),(21;\tau_1(L_3)),(21;\tau_1(L_4))\rbrack,\]
with fixed \(\tau_1(L_1)=((431)^4)\).

\begin{theorem}
\label{thm:c18ES2}

(D.C. Mayer, Sep. \(2015\))

\begin{enumerate}

\item
\(\tau_1(L_2)=(431,\mathbf{(1^3)^3},(1^2)^9)\) and \(\tau_1(L_i)=(431,\mathbf{(21)^3)}\) for \(i=3,4\)
\(\Longleftrightarrow\)\\
\(G\simeq\langle 3^7,285\rangle\mathbf{(-\#1;1)^3}\), of order \(\lvert G\rvert=3^{10}\),
or\\
\(G\simeq\langle 3^7,285\rangle\mathbf{(-\#1;1)^3-\#1;7}\), of order \(\lvert G\rvert=3^{11}\),

\item
\(\tau_1(L_2)=(431,\mathbf{(21^2)^3},(1^2)^9)\) and \(\tau_1(L_i)=(431,\mathbf{(31)^3})\) for \(i=3,4\)
\(\Longleftrightarrow\)\\
\(G\simeq\langle 3^6,49\rangle\mathbf{-\#2;1-\#1;1-\#1;2-\#1;1}\), of order \(\lvert G\rvert=3^{11}\),
or\\
\(G\simeq\langle 3^6,49\rangle\mathbf{-\#2;1-\#1;1-\#1;2-\#1;1-\#1;7}\), of order \(\lvert G\rvert=3^{12}\),
or\\
\(G\simeq\langle 3^6,49\rangle\mathbf{(-\#2;1-\#1;1)^2}\), of order \(\lvert G\rvert=3^{12}\),
or\\
\(G\simeq\langle 3^6,49\rangle\mathbf{(-\#2;1-\#1;1)^2-\#1;1}\), of order \(\lvert G\rvert=3^{13}\),
or\\
\(G\simeq\langle 3^6,49\rangle\mathbf{(-\#2;1-\#1;1)^2-\#1;8}\), of order \(\lvert G\rvert=3^{13}\).

\end{enumerate}

\end{theorem}

\begin{proof}
Similar as in the proof of Theorem
\ref{thm:c18GS},
we construct the descendant tree \(\mathcal{T}(R)\) of the root \(R=\langle 243,6\rangle\),
determine the iterated IPAD of second order \(\tau^{(2)}(V)\) for each vertex \(V\),
and check the second derived quotient \(H/H^{\prime\prime}\) of non-metabelian vertices \(H\).

It turns out that the cover of the common metabelianization
\(\mathfrak{G}=\langle 3^7,285\rangle(-\#1;1)^3=G/G^{\prime\prime}\)
of all candidates \(G\) for the \(3\)-tower group is given by
\(\mathrm{cov}(\langle 3^7,285\rangle(-\#1;1)^3)=\)\\
\(\lbrace\langle 3^7,285\rangle(-\#1;1)^3,\ \langle 3^7,285\rangle(-\#1;1)^3-\#1;7,\)\\
\(\langle 3^6,49\rangle-\#2;1-\#1;1-\#1;2-\#1;1,\ \langle 3^6,49\rangle-\#2;1-\#1;1-\#1;2-\#1;1-\#1;7,\)\\
\(\langle 3^6,49\rangle(-\#2;1-\#1;1)^2,\ \langle 3^6,49\rangle(-\#2;1-\#1;1)^2-\#1;1,\ \langle 3^6,49\rangle(-\#2;1-\#1;1)^2-\#1;8\rbrace\).

The various vertices 
are not characterized uniquely by their iterated IPAD of second order.
Rather they can be identified as batches of two resp. five vertices
in the claimed manner.
\end{proof}

\begin{corollary} 
\label{cor:c18ES2}

(D.C. Mayer, Sep. \(2015\))\\
If \(K\) has torsionfree unit rank \(1\) and does not contain a primitive third root of unity, then

\begin{enumerate}

\item
\(\tau_1(L_2)=(431,\mathbf{(1^3)^3},(1^2)^9)\) and \(\tau_1(L_i)=(431,\mathbf{(21)^3)}\) for \(i=3,4\)
\(\Longleftrightarrow\)\\
\(G\simeq\langle 3^7,285\rangle\mathbf{(-\#1;1)^3-\#1;7}\).

\item
\(\tau_1(L_2)=(431,\mathbf{(21^2)^3},(1^2)^9)\) and \(\tau_1(L_i)=(431,\mathbf{(31)^3})\) for \(i=3,4\)
\(\Longleftrightarrow\)\\
\(G\simeq\langle 3^6,49\rangle\mathbf{-\#2;1-\#1;1-\#1;2-\#1;1-\#1;7}\) 
or\\
\(G\simeq\langle 3^6,49\rangle\mathbf{(-\#2;1-\#1;1)^2-\#1;1}\) 
or\\
\(G\simeq\langle 3^6,49\rangle\mathbf{(-\#2;1-\#1;1)^2-\#1;8}\).

\end{enumerate}

\noindent
All these groups have derived length \(\mathrm{dl}(G)=3\).

\end{corollary}

\begin{proof}
The three groups
\(\langle 3^7,285\rangle\mathbf{(-\#1;1)^3}\)
and
\(\langle 3^6,49\rangle-\mathbf{\#2;1-\#1;1-\#1;2-\#1;1}\)
and
\(\langle 3^6,49\rangle\mathbf{(-\#2;1-\#1;1)^2}\),
which are infinitely capable,
have \(p\)-multiplicator rank \(\mathrm{MR}(G)=4\) and
thus relation rank \(d_2(G)\ge 4\),
and consequently cannot satisfy the Shafarevich inequality \(d_2(G)\le d_1(G)+r=2+1=3\)
\cite[Thm.6, p.140]{Sh}
for a field \(K\) with unit rank \(r=r_1+r_2-1=1\) and \(\zeta_3\notin K\).
\end{proof}



\subsection{Real Quadratic Fields of Type c.18\(\uparrow^2\)}
\label{ss:RQFc18ES2}

\begin{proposition}
\label{prp:RQFc18ES2}

(M.R. Bush, Jul. \(2015\))\\
In the range \(0<d<10^9\) of fundamental discriminants \(d\)
of real quadratic fields \(K=\mathbb{Q}(\sqrt{d})\)
there exist precisely \(\mathbf{5}\) cases with \(1^{\mathrm{st}}\) IPAD
\(\tau^{(1)}(K)=\lbrack 1^2;\mathbf{4^2},1^3,(21)^2\rbrack\).

\end{proposition}

\begin{proof}
The results were communicated to us on July 11, 2015, by M.R. Bush, who used PARI/GP
\cite{PARI}
with similar techniques as described in our paper
\cite[\S\ 5, pp.446--450]{Ma3},
and additionally double-checked with MAGMA
\cite{MAGMA}.
\end{proof}

\begin{corollary}
\label{cor:RQFc18ES2}

(D.C. Mayer, \(2010\))\\
A quadratic field \(K=\mathbb{Q}(\sqrt{d})\)
with
\(\tau^{(1)}(K)=\lbrack 1^2;\mathbf{4^2},1^3,(21)^2\rbrack\)
must be a real quadratic field
with \(3\)-capitulation type \(\varkappa_1(K)=(0122)\).

\end{corollary}

\begin{proof}
This is a consequence of item (1) in
\cite[Cor.4.4.3, p.442]{Ma3}.
\end{proof}

\begin{theorem} 
\label{thm:RQFc18ES2}

(D.C. Mayer, Sep. \(2015\))

\begin{enumerate}

\item
The \textbf{single} real quadratic fields \(K=\mathbb{Q}(\sqrt{d})\)
with discriminant \(d=241\,798\,776\) (\(\mathbf{20}\%\) of \(5\)),
has \(3\)-class tower group
\(G\simeq\langle 3^7,285\rangle\mathbf{(-\#1;1)^3-\#1;7}\).

\item
The \(\mathbf{4}\) real quadratic fields \(K=\mathbb{Q}(\sqrt{d})\) with the following discriminants \(d\) (\(\mathbf{80}\%\) of \(5\)),

\begin{center}
\(174\,458\,681,\ 298\,160\,513,\ 496\,930\,117,\ 743\,138\,141,\)
\end{center}

\noindent
have \(3\)-class tower group either\\
\(G\simeq\langle 3^6,49\rangle\mathbf{-\#2;1-\#1;1-\#1;2-\#1;1-\#1;7}\) or\\
\(G\simeq\langle 3^6,49\rangle\mathbf{(-\#2;1-\#1;1)^2-\#1;1}\) or\\
\(G\simeq\langle 3^6,49\rangle\mathbf{(-\#2;1-\#1;1)^2-\#1;8}\).

\end{enumerate}

\noindent
In each case, the length of the \(3\)-class tower of \(K\) is given by \(\ell_3(K)=3\).

\end{theorem}

\begin{proof}
Since all these real quadratic fields \(K=\mathbb{Q}(\sqrt{d})\) have
\(3\)-capitulation type \(\varkappa_1(K)=(0122)\), \(1^{\mathrm{st}}\) IPAD  
\(\tau^{(1)}(K)=\lbrack 1^2;\mathbf{4^2},1^3,(21)^2\rbrack\)
and suitable \(2^{\mathrm{nd}}\) IPAD,
the claim is a consequence of Corollary
\ref{cor:c18ES2}.
\end{proof}



\section{Capitulation Type c.21, \((2034)\)}
\label{s:c21}

\subsection{The Ground State of Capitulation Type c.21}
\label{ss:c21GS}

Assume that \(K\) is an algebraic number field
with \(3\)-class group \(\mathrm{Cl}_3(K)\) of type \((3,3)\hat{=}1^2\).
Let \(L_1,\ldots,L_4\) be the four unramified cyclic cubic extensions of \(K\),
and suppose that the \(1^{\mathrm{st}}\) order Artin pattern \(\mathrm{AP}^{(1)}(K)\) of \(K\)
is given by the IPAD 
\(\tau^{(1)}(K)=\lbrack 1^2;21,\mathbf{2^2},21,21\rbrack\)
and the IPOD
\(\varkappa^{(1)}(K)=\lbrack G^\prime;H_2,\mathbf{G},H_3,H_4\rbrack\),
i.e. \(\tau_1(K)=((9,3),(9,9),(9,3),(9,3))\) are the type invariants of the \(3\)-class groups of the \(L_i\)
and \(\varkappa_1(K)=(2034)\) is the \(3\)-capitulation type of \(K\) in the \(L_i\).

\begin{proposition} (D.C. Mayer, \(2010\))
\label{prp:c21GS}

\begin{enumerate}

\item
The second \(3\)-class group \(\mathfrak{G}=\mathrm{G}_3^2(K)\) of \(K\)
is isomorphic to the metabelian \(3\)-group \(\langle 729,54\rangle\)
and thus its relation rank satisfies the inequality \(d_2(\mathfrak{G})\ge d_1(\mathfrak{G})+2\).

\item
Consequently,
if \(K\) is a number field with torsionfree unit rank \(r=1\)
and does not contain the third roots of unity
(in particular, if \(K=\mathbb{Q}(\sqrt{d})\) is a real quadratic field),
then the length of the \(3\)-class field tower of \(K\) must be \(\ell_3(K)\ge 3\).

\end{enumerate}

\end{proposition}

\begin{proof}

\begin{enumerate}

\item
First, we prove that \(\mathfrak{G}\) must have coclass \(\mathrm{cc}(\mathfrak{G})=2\).
This can be done in two ways,
either using \(\tau_1(K)\) alone or using \(\varkappa_1(K)\) alone.

\begin{itemize}
\item
According to items 1) and 3) of
\cite[Thm.3.2, p.291]{Ma6},
\(\tau_1(\mathfrak{G})=\tau_1(K)\)
must contain three components of type \(1^2\) if \(\mathrm{cc}(\mathfrak{G})=1\),
and two components of type \(1^3\) if \(\mathrm{cc}(\mathfrak{G})\ge 3\),
whence \(\mathrm{cc}(\mathfrak{G})=2\) is the only possibility for
\(\tau_1(\mathfrak{G})=(21,2^2,(21)^2)\) without any \(1^2\) and without any \(1^3\).
\item
According to
\cite[Thm.2.5, p.479]{Ma2},
\(\varkappa_1(\mathfrak{G})=\varkappa_1(K)\) must contain three total kernels,
designated by \(0\), if \(\mathrm{cc}(\mathfrak{G})=1\).
Since all metabelian \(3\)-groups of coclass bigger than \(2\)
are descendants of \(\langle 3^5,3\rangle\),
and the TKT \((2100)\) of this root contains a \(2\)-cycle,
the TKT \(\varkappa_1(\mathfrak{G})\) must also contain a \(2\)-cycle,
if \(\mathrm{cc}(\mathfrak{G})\ge 3\),
according to
\cite[Cor.3.0.2, p.772]{BuMa}.
Therefore, the only possibility for \(\varkappa_1(\mathfrak{G})=(2034)\)
without a \(2\)-cycle and with only one \(0\)
is \(\mathrm{cc}(\mathfrak{G})=2\).
\end{itemize}

Next, we show that \(\mathfrak{G}\) must be a descendant of \(\langle 3^5,8\rangle\).
According to item 2) of
\cite[Thm.3.1, p.290]{Ma6},
the polarized component \(2^2\) of order \(3^4\) of \(\tau_1(\mathfrak{G})=\tau_1(K)\)
cannot occur for a sporadic \(3\)-group outside of coclass-\(2\) trees,
whence \(\mathfrak{G}\) must be a vertex of one of the three coclass-\(2\) trees
with metabelian mainline.
Their roots are
\(\langle 3^5,3\rangle\), resp. \(\langle 3^5,6\rangle\), resp. \(\langle 3^5,8\rangle\).
All descendants of these roots have three stable components of their TTT,
\(((1^3)^2,21)\), resp. \((1^3,(21)^2)\), resp. \(((21)^3)\),
whence \(\tau_1(\mathfrak{G})=(21,2^2,(21)^2)\)
unambiguously leads to a descendant of \(\langle 3^5,8\rangle\), by item 2) of
\cite[Thm.3.2, p.291]{Ma6}.
Furthermore, we have
\(d_2(\mathfrak{G})\ge\mathrm{MR}(\mathfrak{G})=4=d_1(\mathfrak{G})+2\),
since the \(p\)-multiplicator rank is a lower bound for the relation rank.

Finally, the Artin pattern
\(\mathrm{AP}(\mathfrak{G})=(\varkappa_1(\mathfrak{G}),\tau_1(\mathfrak{G}))\)
provides a sort of coordinate system
in which the coclass tree with root \(\langle 3^5,8\rangle\) is embedded,
with horizontal axis \(\varkappa_1(\mathfrak{G})\) and vertical axis \(\tau_1(\mathfrak{G})\).
The polarization \(2^2\) of \(\tau_1(\mathfrak{G})=(21,2^2,(21)^2)\)
determines the nilpotency class \(\mathrm{cl}(\mathfrak{G})=4\)
and thus also the order \(\lvert\mathfrak{G}\rvert=3^6\), according to item 2) of
\cite[Thm.3.2, p.291]{Ma6},
since the defect \(k=1\) is only possible for type G.16 \((2134)\).
The polarization \(0\) of \(\varkappa_1(\mathfrak{G})=(2034)\) with stable components \((234)\)
unambiguously identifies the mainline vertex of order \(3^6\),
which is \(\langle 3^6,54\rangle\), according to Figure
\ref{fig:TreeOverviewU}.

\item
According to
\cite[Thm.6, p.140]{Sh},
the relation rank \(d_2(G)\) of the \(3\)-class tower group
\(G:=\mathrm{Gal}(\mathrm{F}_3^\infty(K)\vert K)\)
of a number field \(K\) with torsionfree unit rank \(r=r_1+r_2-1=1\) and \(\zeta_3\notin K\)
must satisfy \(d_2(G)\le d_1(G)+r=2+1=3\),
whence \(G\not\simeq\mathfrak{G}\) and \(\ell_3(K)=\mathrm{dl}(G)\ge 3\).
\end{enumerate}
\end{proof}

\noindent
Suppose now that, under the assumptions preceding Proposition
\ref{prp:c21GS},
\(G=\mathrm{G}_3^\infty(K)\) denotes the \(3\)-class tower group of \(K\) and
we are additionally given the \(2^{\mathrm{nd}}\) order IPAD of \(K\),
\[\tau^{(2)}(K)=
\lbrack 1^2;(21;\tau_1(L_1)),(\mathbf{2^2};\tau_1(L_2)),(21;\tau_1(L_3)),(21;\tau_1(L_4))\rbrack,\]
with fixed \(\tau_1(L_1)=(21^2,(21)^3)\) and \(\tau_1(L_2)=((21^2)^4)\).

\begin{theorem}
\label{thm:c21GS}

(D.C. Mayer, Aug. \(2015\))\\
Among the non-metabelian candidates for the \(3\)-tower group \(G\) of \(K\),
the immediate descendants of step size \(1\) of \(\langle 729,54\rangle\)
are characterized by the following criteria:

\begin{enumerate}

\item
\(\tau_1(L_3)=(21^2,(21)^3)\), \(\tau_1(L_4)=(21^2,\mathbf{(31)^3})\)
\(\Longleftrightarrow\)
\(G\simeq\langle 2187,\mathbf{307}\rangle\),

\item
\(\tau_1(L_3)=(21^2,\mathbf{(31)^3})\), \(\tau_1(L_4)=(21^2,(21)^3)\) 
\(\Longleftrightarrow\)
\(G\simeq\langle 2187,\mathbf{308}\rangle\).

\end{enumerate}

\noindent
In both cases, we have derived length \(\mathrm{dl}(G)=3\) and nilpotency class \(\mathrm{cl}(G)=5\).

\end{theorem}

\begin{remark}
\label{rmk:c21GS}

From the viewpoint of group theory,
where a strict natural ordering can be imposed on the maximal subgroups of \(G\) in terms of generators,
the conditions of the two statements in Theorem
\ref{thm:c21GS}
are distinct.
From the viewpoint of number theory, however,
no canonical ordering exists and the conditions are indistinguishable.

\end{remark}

\begin{proof}
In analogy to the proof of Theorem
\ref{thm:c18GS},
we construct the descendant tree \(\mathcal{T}(R)\) of the root \(R=\langle 243,8\rangle\),
which is restricted to the coclass tree \(\mathcal{T}^2(R)\) in Figure
\ref{fig:TreeOverviewU}
by ignoring the bifurcation at the not coclass-settled vertex \(\langle 729,54\rangle\).
In Figure
\ref{fig:C21PrunedTreeU},
all periodic bifurcations
\cite[\S\ 21.2]{Ma5}
in the complete descendant tree are taken into consideration,
but the tree is pruned from all TKTs different from c.21.
In parallel computation with the recursive tree construction,
the iterated IPAD of second order \(\tau^{(2)}(V)\) is determined for each vertex \(V\) and
non-metabelian vertices \(H\) are checked for their second derived quotient \(H/H^{\prime\prime}\).
The construction can be terminated at order \(3^{11}\), because
several components of the \(2^{\mathrm{nd}}\) order IPAD become stable
and the remaining components reveal a deterministic growth: we have
\[\tau_1(L_1)=\tau_1(L_3)=\tau_1(L_4)=(\ast,\mathbf{(21)^3})\text{ for vertices of coclass }\mathbf{2}\]
and
\[\tau_1(L_1)=\tau_1(L_3)=\tau_1(L_4)=(\ast,\mathbf{(31)^3})\text{ for vertices of coclass }\mathbf{3}.\]
Therefore, the vertices \(\langle 2187,307\rangle\) and \(\langle 2187,308\rangle\)
are characterized uniquely by their iterated IPAD of second order,
and the cover of their common parent is given by
\[\mathrm{cov}(\langle 729,54\rangle)=\lbrace\langle 729,54\rangle,\langle 2187,307\rangle,\langle 2187,308\rangle\rbrace.\]
\end{proof}



\subsection{Real Quadratic Fields of Type c.21}
\label{ss:RQFc21GS}

\begin{proposition}
\label{prp:RQFc21GS}

(D.C. Mayer, Feb. \(2010\))\\
In the range \(0<d<10^7\) of fundamental discriminants \(d\)
of real quadratic fields \(K=\mathbb{Q}(\sqrt{d})\)
there exist precisely \(\mathbf{27}\) cases
with \(3\)-capitulation type \(\varkappa_1(K)=(2034)\).
The \(1^{\mathrm{st}}\) IPAD of \(\mathbf{25}\) among them 
is \(\tau^{(1)}(K)=\lbrack 11;21,\mathbf{22},21,21\rbrack\),
the remaining \(2\) cases have \(\tau^{(1)}(K)=\lbrack 11;21,\mathbf{33},21,21\rbrack\)
(and will be considered in section \S\
\ref{ss:c21ES1}).

\end{proposition}

\begin{proof}
The results
\cite[Tbl.6.7, p.453]{Ma3}
were computed by means of the free number theoretic computer algebra system PARI/GP
\cite{PARI}
using an implementation of our own principalization algorithm in a PARI script, as described in detail in
\cite[\S\ 5, pp.446--450]{Ma3}.
\end{proof}

\begin{theorem} (D.C. Mayer, Aug. \(2015\))
\label{thm:RQFc21GS}

\noindent
The \(\mathbf{25}\) real quadratic fields \(K=\mathbb{Q}(\sqrt{d})\) with the following discriminants \(d\),
\[
\begin{aligned}
   540\,365 &,&    945\,813 &,& 1\,202\,680 &,& 1\,695\,260 &,& 1\,958\,629, \\
3\,018\,569 &,& 3\,236\,657 &,& 3\,687\,441 &,& 4\,441\,560 &,& 5\,512\,252, \\
5\,571\,377 &,& 5\,701\,693 &,& 6\,027\,557 &,& 6\,049\,356 &,& 6\,054\,060, \\
6\,274\,609 &,& 6\,366\,029 &,& 6\,501\,608 &,& 6\,773\,557 &,& 7\,573\,868, \\
8\,243\,464 &,& 8\,251\,521 &,& 9\,054\,177 &,& 9\,162\,577 &,& 9\,967\,837, 
\end{aligned}
\]

\noindent
have \(3\)-class tower group either
\(G\simeq\langle 3^7,\mathbf{307}\rangle\) or
\(G\simeq\langle 3^7,\mathbf{308}\rangle\).

\noindent
In both cases, the length of the \(3\)-class tower of \(K\) is given by \(\ell_3(K)=3\).

\end{theorem}

\begin{proof}
Since all these real quadratic fields \(K=\mathbb{Q}(\sqrt{d})\) have
\(3\)-capitulation type \(\varkappa_1(K)=(2034)\) and \(1^{\mathrm{st}}\) IPAD  
\(\tau^{(1)}(K)=\lbrack 1^2;21,\mathbf{2^2},(21)^2\rbrack\),
and all \(25\) fields have unordered and indistinguishable \(2^{\mathrm{nd}}\) IPAD
\[\tau_1(L_1)=(21^2,\mathbf{(31)^3}),\ \tau_1(L_2)=((21^2)^4),\ \tau_1(L_3)=(21^2,(21)^3),\ \tau_1(L_4)=(21^2,(21)^3),\]
the claim is a consequence of Theorem
\ref{thm:c21GS}.
\end{proof}

\begin{remark}
\label{rmk:RQFc21GS}
Unpublished results of M.R. Bush show that there are \(4377\) real quadratic fields
with discriminants \(0<d<10^9\) having the first IPAD in Proposition
\ref{prp:RQFc21GS}.
For the remaining \(4352\) cases outside of the range \(0<d<10^7\),
we can only identify the \(3\)-tower group up to the two possibilities in Theorem
\ref{thm:RQFc21GS}.
However, according to Proposition
\ref{prp:c21GS}
and Theorem
\ref{thm:c21GS},
we know that the length of the \(3\)-tower is \(\ell_3(K)=3\). 
\end{remark}



\noindent
Figure
\ref{fig:C21PrunedTreeU}
visualizes the groups in Theorems
\ref{thm:c21GS}
and
\ref{thm:c21ES1}
and their population in Theorems
\ref{thm:RQFc21GS}
and
\ref{thm:RQFc21ES1}.

{\tiny

\begin{figure}[hb]
\caption{Non-metabelian \(3\)-tower groups \(G\) on the pruned tree \(\mathcal{T}_\ast(\langle 243,8\rangle)\)}
\label{fig:C21PrunedTreeU}


\setlength{\unitlength}{0.8cm}
\begin{picture}(18,22)(-6,-21)

\put(-5,0.5){\makebox(0,0)[cb]{Order}}
\put(-5,0){\line(0,-1){18}}
\multiput(-5.1,0)(0,-2){10}{\line(1,0){0.2}}
\put(-5.2,0){\makebox(0,0)[rc]{\(243\)}}
\put(-4.8,0){\makebox(0,0)[lc]{\(3^5\)}}
\put(-5.2,-2){\makebox(0,0)[rc]{\(729\)}}
\put(-4.8,-2){\makebox(0,0)[lc]{\(3^6\)}}
\put(-5.2,-4){\makebox(0,0)[rc]{\(2\,187\)}}
\put(-4.8,-4){\makebox(0,0)[lc]{\(3^7\)}}
\put(-5.2,-6){\makebox(0,0)[rc]{\(6\,561\)}}
\put(-4.8,-6){\makebox(0,0)[lc]{\(3^8\)}}
\put(-5.2,-8){\makebox(0,0)[rc]{\(19\,683\)}}
\put(-4.8,-8){\makebox(0,0)[lc]{\(3^9\)}}
\put(-5.2,-10){\makebox(0,0)[rc]{\(59\,049\)}}
\put(-4.8,-10){\makebox(0,0)[lc]{\(3^{10}\)}}
\put(-5.2,-12){\makebox(0,0)[rc]{\(177\,147\)}}
\put(-4.8,-12){\makebox(0,0)[lc]{\(3^{11}\)}}
\put(-5.2,-14){\makebox(0,0)[rc]{\(531\,441\)}}
\put(-4.8,-14){\makebox(0,0)[lc]{\(3^{12}\)}}
\put(-5.2,-16){\makebox(0,0)[rc]{\(1\,594\,323\)}}
\put(-4.8,-16){\makebox(0,0)[lc]{\(3^{13}\)}}
\put(-5.2,-18){\makebox(0,0)[rc]{\(4\,782\,969\)}}
\put(-4.8,-18){\makebox(0,0)[lc]{\(3^{14}\)}}
\put(-5,-18){\vector(0,-1){2}}

\put(0.1,0.2){\makebox(0,0)[lb]{\(\langle 8\rangle\)}}
\put(0.1,-1.8){\makebox(0,0)[lb]{\(\langle 54\rangle\) (not coclass-settled)}}
\put(1.1,-2.8){\makebox(0,0)[lb]{\(1^{\text{st}}\) bifurcation}}
\put(0.1,-3.8){\makebox(0,0)[lb]{\(\langle 303\rangle\)}}
\put(0.1,-5.8){\makebox(0,0)[lb]{\(1;1\)}}
\put(0.1,-7.8){\makebox(0,0)[lb]{\(1;1\)}}
\put(0.1,-9.8){\makebox(0,0)[lb]{\(1;1\)}}
\put(0.1,-11.8){\makebox(0,0)[lb]{\(1;1\)}}
\multiput(0,0)(0,-2){7}{\circle*{0.2}}
\multiput(0,0)(0,-2){6}{\line(0,-1){2}}
\put(0,-12){\vector(0,-1){2}}
\put(-0.2,-14.2){\makebox(0,0)[rt]{\(\mathcal{T}_\ast^2(\langle 243,8\rangle)\)}}

\put(-2,-4.2){\makebox(0,0)[ct]{\(\langle 308\rangle\)}}
\put(-2,-8.2){\makebox(0,0)[ct]{\(1;7\)}}
\put(-2,-12.2){\makebox(0,0)[ct]{\(1;7\)}}
\multiput(0,-2)(0,-4){3}{\line(-1,-1){2}}
\multiput(-2.05,-4.05)(0,-4){3}{\framebox(0.1,0.1){}}

\put(-1,-4.2){\makebox(0,0)[ct]{\(\langle 307\rangle\)}}
\multiput(0,-2)(0,-4){1}{\line(-1,-2){1}}
\multiput(-1.05,-4.05)(0,-4){1}{\framebox(0.1,0.1){}}

\put(0,-2){\line(1,-1){4}}

\put(4.1,-5.8){\makebox(0,0)[lb]{\(2;3\)}}
\put(4.1,-7.8){\makebox(0,0)[lb]{\(1;1\) (not coclass-settled)}}
\put(5.1,-8.8){\makebox(0,0)[lb]{\(2^{\text{nd}}\) bifurcation}}
\put(4.1,-9.8){\makebox(0,0)[lb]{\(1;1\)}}
\put(4.1,-11.8){\makebox(0,0)[lb]{\(1;1\)}}
\put(4.1,-13.8){\makebox(0,0)[lb]{\(1;1\)}}
\multiput(3.95,-6.05)(0,-2){5}{\framebox(0.1,0.1){}}
\multiput(4,-6)(0,-2){4}{\line(0,-1){2}}
\put(4,-14){\vector(0,-1){2}}
\put(3.8,-16.2){\makebox(0,0)[rt]{\(\mathcal{T}_\ast^3(\langle 729,54\rangle-\#2;3)\)}}

\put(2,-10.2){\makebox(0,0)[ct]{\(1;8\)}}
\put(2,-14.2){\makebox(0,0)[ct]{\(1;7\)}}
\multiput(4,-8)(0,-4){2}{\line(-1,-1){2}}
\multiput(1.95,-10.05)(0,-4){2}{\framebox(0.1,0.1){}}

\put(3,-10.2){\makebox(0,0)[ct]{\(1;7\)}}
\multiput(4,-8)(0,-4){1}{\line(-1,-2){1}}
\multiput(2.95,-10.05)(0,-4){1}{\framebox(0.1,0.1){}}

\put(4,-8){\line(1,-1){4}}

\put(8.1,-11.8){\makebox(0,0)[lb]{\(2;1\)}}
\put(8.1,-13.8){\makebox(0,0)[lb]{\(1;1\) (not coclass-settled)}}
\put(9.1,-14.8){\makebox(0,0)[lb]{\(3^{\text{rd}}\) bifurcation}}
\put(8.1,-15.8){\makebox(0,0)[lb]{\(1;1\)}}
\multiput(7.95,-12.05)(0,-2){3}{\framebox(0.1,0.1){}}
\multiput(8,-12)(0,-2){2}{\line(0,-1){2}}
\put(8,-16){\vector(0,-1){2}}
\put(7.8,-18.2){\makebox(0,0)[rt]{\(\mathcal{T}_\ast^4(\langle 729,54\rangle-\#2;3-\#1;1-\#2;1)\)}}

\put(6,-16.2){\makebox(0,0)[ct]{\(1;8\)}}
\multiput(8,-14)(0,-4){1}{\line(-1,-1){2}}
\multiput(5.95,-16.05)(0,-4){1}{\framebox(0.1,0.1){}}

\put(7,-16.2){\makebox(0,0)[ct]{\(1;7\)}}
\multiput(8,-14)(0,-4){1}{\line(-1,-2){1}}
\multiput(6.95,-16.05)(0,-4){1}{\framebox(0.1,0.1){}}

\put(8,-14){\line(1,-1){4}}

\put(12.1,-17.8){\makebox(0,0)[lb]{\(2;1\)}}
\multiput(11.95,-18.05)(0,-2){1}{\framebox(0.1,0.1){}}
\put(12,-18){\vector(0,-1){2}}
\put(11.8,-19.9){\makebox(0,0)[rt]{\(\mathcal{T}_\ast^5(\langle 729,54\rangle-\#2;3-\#1;1-\#2;1-\#1;1-\#2;1)\)}}



\multiput(0.3,-2.2)(0,-4){3}{\oval(1,1.5)}
\put(-0.6,-2.1){\makebox(0,0)[rc]{\underbar{\textbf{\(\#4377\)}}}}
\put(-0.6,-6.1){\makebox(0,0)[rc]{\underbar{\textbf{\(\#146\)}}}}
\put(-0.6,-10.1){\makebox(0,0)[rc]{\underbar{\textbf{\(\#5\)}}}}

\multiput(-2,-8.25)(0,-4){2}{\oval(1,1.5)}
\put(-2,-9.3){\makebox(0,0)[cc]{\underbar{\textbf{\(+1\,001\,957\)}}}}
\put(-2,-13.3){\makebox(0,0)[cc]{\underbar{\textbf{\(+407\,086\,012\)}}}}

\multiput(-1.5,-4.25)(0,-4){1}{\oval(2,1.5)}
\put(-1.5,-5.3){\makebox(0,0)[cc]{\underbar{\textbf{\(+540\,365\)}}}}

\multiput(2.5,-10.25)(4,-6){2}{\oval(2,1.5)}
\put(2.5,-11.3){\makebox(0,0)[cc]{\underbar{\textbf{\(+25\,283\,701\)}}}}

\put(3,-15){\vector(-1,1){0.5}}
\put(4,-15.3){\makebox(0,0)[cc]{\underbar{\textbf{\(+116\,043\,324\)}}}}
\put(5,-15.5){\vector(1,-1){0.5}}

\multiput(2,-14.25)(4,-6){1}{\oval(1,1.5)}

\end{picture}

\end{figure}

}



\subsection{The 1\({}^{\text{st}}\) Excited State of Capitulation Type c.21}
\label{ss:c21ES1}

As before, \(K\) is an algebraic number field
with \(3\)-class group \(\mathrm{Cl}_3(K)\) of type \((3,3)\hat{=}1^2\).
Let \(L_1,\ldots,L_4\) be the four unramified cyclic cubic extensions of \(K\),
and suppose that the \(1^{\mathrm{st}}\) order Artin pattern \(\mathrm{AP}^{(1)}(K)\) of \(K\)
is given by the IPAD 
\(\tau^{(1)}(K)=\lbrack 1^2;21,\mathbf{3^2},21,21\rbrack\)
and the IPOD
\(\varkappa^{(1)}(K)=\lbrack G^\prime;H_2,\mathbf{G},H_3,H_4\rbrack\),
i.e. \(\tau_1(K)=((9,3),(27,27),(9,3),(9,3))\) are the type invariants of the \(3\)-class groups of the \(L_i\)
and \(\varkappa_1(K)=(2034)\) is the \(3\)-capitulation type of \(K\) in the \(L_i\).

\begin{proposition}
\label{prp:c21ES1}

(D.C. Mayer, \(2010\))

\begin{enumerate}

\item
The second \(3\)-class group \(\mathfrak{G}=\mathrm{G}_3^2(K)\) of \(K\)
is isomorphic to the metabelian \(3\)-group \(\langle 2187,303\rangle-\#1;1\)
and thus its relation rank satisfies the inequality
\(d_2(\mathfrak{G})\ge d_1(\mathfrak{G})+2\).

\item
Consequently,
if \(K\) is a number field with torsionfree unit rank \(r=1\)
and does not contain the third roots of unity
(in particular, if \(K=\mathbb{Q}(\sqrt{d})\) is a real quadratic field),
then the length of the \(3\)-class field tower of \(K\) must be \(\ell_3(K)\ge 3\).

\end{enumerate}

\end{proposition}

\begin{proof}

A great deal of the proof is similar to the proof of Proposition
\ref{prp:c21GS}.

\begin{enumerate}

\item
First, we prove that \(\mathfrak{G}\) must have coclass \(\mathrm{cc}(\mathfrak{G})=2\).
Next, we show that \(\mathfrak{G}\) must be a descendant of \(\langle 3^5,8\rangle\),
this time using the polarized component \(3^2\) of order \(3^6\) of \(\tau_1(\mathfrak{G})\).
Again, we have
\(d_2(\mathfrak{G})\ge\mathrm{MR}(\mathfrak{G})=4=d_1(\mathfrak{G})+2\),
since the \(p\)-multiplicator rank is a lower bound for the relation rank.
Finally, the polarization \(3^2\) of \(\tau_1(\mathfrak{G})=(21,3^2,(21)^2)\)
determines the nilpotency class \(\mathrm{cl}(\mathfrak{G})=6\)
and thus also the order \(\lvert\mathfrak{G}\rvert=3^8\).
The polarization \(0\) of \(\varkappa_1(\mathfrak{G})=(2034)\) with stable components \((234)\)
unambiguously identifies the mainline vertex of order \(3^8\),
which is \(\langle 2187,303\rangle-\#1;1\), according to Figure
\ref{fig:C21PrunedTreeU}.

\item
As before,
\cite[Thm.6, p.140]{Sh}
implies that the relation rank \(d_2(G)\) of the \(3\)-class tower group
\(G:=\mathrm{Gal}(\mathrm{F}_3^\infty(K)\vert K)\)
of a number field \(K\) with unit rank \(r=r_1+r_2-1=1\) and \(\zeta_3\notin K\)
satisfies \(d_2(G)\le d_1(G)+r=2+1=3\),
whence \(G\not\simeq\mathfrak{G}\) and \(\ell_3(K)=\mathrm{dl}(G)\ge 3\).
\end{enumerate}
\end{proof}

\noindent
Suppose now that, under the assumptions preceding Proposition
\ref{prp:c21ES1},
\(G=\mathrm{G}_3^\infty(K)\) denotes the \(3\)-class tower group of \(K\) and
we are additionally given the \(2^{\mathrm{nd}}\) order IPAD of \(K\),
\[\tau^{(2)}(K)=
\lbrack 1^2;((21;\tau_1(L_1)),(\mathbf{3^2};\tau_1(L_2)),(21;\tau_1(L_3)),(21;\tau_1(L_4))\rbrack,\]
with fixed \(\tau_1(L_2)=((321)^4)\).


\begin{theorem} (D.C. Mayer, Aug. \(2015\))
\label{thm:c21ES1}

\begin{enumerate}

\item
\(\tau_1(L_i)=(321,\mathbf{(21)^3)}\) for \(i=1,3,4\)
\(\Longleftrightarrow\)\\
\(G\simeq\langle 3^7,303\rangle-\mathbf{\#1;1}\), of order \(\lvert G\rvert=3^8\),
or\\
\(G\simeq\langle 3^7,303\rangle-\mathbf{\#1;1-\#1;7}\), of order \(\lvert G\rvert=3^9\).

\item
\(\tau_1(L_i)=(321,\mathbf{(31)^3})\) for \(i=1,3,4\)
\(\Longleftrightarrow\)\\
\(G\simeq\langle 3^6,54\rangle-\mathbf{\#2;3-\#1;1}\), of order \(\lvert G\rvert=3^9\),
or\\
\(G\simeq\langle 3^6,54\rangle-\mathbf{\#2;3-\#1;1-\#1;7}\), of order \(\lvert G\rvert=3^{10}\),
or\\
\(G\simeq\langle 3^6,54\rangle-\mathbf{\#2;3-\#1;1-\#1;8}\), of order \(\lvert G\rvert=3^{10}\).

\end{enumerate}

\end{theorem}

\begin{proof}
Similar as in the proof of Theorem
\ref{thm:c21GS},
we construct the descendant tree \(\mathcal{T}(R)\) of the root \(R=\langle 243,8\rangle\),
determine the iterated IPAD of second order \(\tau^{(2)}(V)\) for each vertex \(V\),
and check the second derived quotient \(H/H^{\prime\prime}\) of non-metabelian vertices \(H\).

It turns out that the cover of the common metabelianization
\(\mathfrak{G}=\langle 3^7,303\rangle-\#1;1=G/G^{\prime\prime}\)
of all candidates \(G\) for the \(3\)-tower group is given by
\(\mathrm{cov}(\langle 3^7,303\rangle-\#1;1)=\)\\
\(\lbrace\langle 3^7,303\rangle-\#1;1,\ \langle 3^7,303\rangle-\#1;1-\#1;7,\)\\
\(\langle 3^6,54\rangle-\#2;3-\#1;1,\ \langle 3^6,54\rangle-\#2;3-\#1;1-\#1;7,\ \langle 3^6,54\rangle-\#2;3-\#1;1-\#1;8\rbrace\).

The various vertices 
are not characterized uniquely by their iterated IPAD of second order.
Rather they can be identified as batches of two resp. three vertices
in the claimed manner.
\end{proof}

\begin{corollary} 
\label{cor:c21ES1}

(D.C. Mayer, Aug. \(2015\))\\
If \(K\) has torsionfree unit rank \(1\) and does not contain a primitive third root of unity, then

\begin{enumerate}

\item
\(\tau_1(L_i)=(321,\mathbf{(21)^3)}\) for \(i=1,3,4\)
\(\Longleftrightarrow\)\\
\(G\simeq\langle 3^7,303\rangle-\mathbf{\#1;1-\#1;7}\).

\item
\(\tau_1(L_i)=(321,\mathbf{(31)^3})\) for \(i=1,3,4\)
\(\Longleftrightarrow\)\\
\(G\simeq\langle 3^6,54\rangle-\mathbf{\#2;3-\#1;1-\#1;7}\) or\\
\(G\simeq\langle 3^6,54\rangle-\mathbf{\#2;3-\#1;1-\#1;8}\).

\end{enumerate}

\noindent
All these groups have derived length \(\mathrm{dl}(G)=3\).

\end{corollary}

\begin{proof}
The two infinitely capable groups
\(\langle 3^7,303\rangle-\mathbf{\#1;1}\)
and
\(\langle 3^6,54\rangle-\mathbf{\#2;3-\#1;1}\)
have \(p\)-multiplicator rank \(\mathrm{MR}(G)=4\) and
thus relation rank \(d_2(G)\ge 4\),
and consequently cannot satisfy the Shafarevich inequality \(d_2(G)\le d_1(G)+r=2+1=3\)
\cite[Thm.6, p.140]{Sh}
for a field \(K\) with unit rank \(r=r_1+r_2-1=1\) and \(\zeta_3\notin K\).
\end{proof}



\subsection{Real Quadratic Fields of Type c.21\(\uparrow\)}
\label{ss:RQFc21ES1}

\begin{proposition} 
\label{prp:RQFc21ES1}

(M.R. Bush, Jul. \(2015\))\\
In the range \(0<d<10^8\) of fundamental discriminants \(d\)
of real quadratic fields \(K=\mathbb{Q}(\sqrt{d})\)
there exist precisely \(\mathbf{12}\) cases with \(1^{\mathrm{st}}\) IPAD
\(\tau^{(1)}(K)=\lbrack 1^2;21,\mathbf{3^2},(21)^2\rbrack\).

\end{proposition}

\begin{proof}
The results were communicated to us on July 11, 2015, by M.R. Bush, who used PARI/GP
\cite{PARI}
with similar techniques as described in our paper
\cite[\S\ 5, pp.446--450]{Ma3},
and additionally double-checked with MAGMA
\cite{MAGMA}.
\end{proof}

\begin{corollary} 
\label{cor:RQFc21ES1}

(D.C. Mayer, \(2010\))\\
A quadratic field \(K=\mathbb{Q}(\sqrt{d})\)
with
\(\tau^{(1)}(K)=\lbrack 1^2;21,\mathbf{3^2},(21)^2\rbrack\)
must be a real quadratic field
with \(3\)-capitulation type \(\varkappa_1(K)=(2034)\).

\end{corollary}

\begin{proof}
This is a consequence of item (1) in
\cite[Cor.4.4.3, p.442]{Ma3}.
\end{proof}

\bigskip
\noindent
\begin{theorem} 
\label{thm:RQFc21ES1}

(D.C. Mayer, Aug. \(2015\))

\begin{enumerate}

\item
The \(\mathbf{8}\) real quadratic fields \(K=\mathbb{Q}(\sqrt{d})\) with the following discriminants \(d\) (\(\mathbf{67}\%\) of \(12\)),
\[
\begin{aligned}
 1\,001\,957 &,&  9\,923\,685 &,& 20\,633\,209 &,& 58\,650\,717, \\  
63\,404\,792 &,& 72\,410\,413 &,& 84\,736\,636 &,& 92\,578\,472, 
\end{aligned}
\]

\noindent
have \(3\)-class tower group
\(G\simeq\langle 3^7,303\rangle-\mathbf{\#1;1-\#1;7}\).

\item
The \(\mathbf{4}\) real quadratic fields \(K=\mathbb{Q}(\sqrt{d})\) with the following discriminants \(d\) (\(\mathbf{33}\%\) of \(12\)),

\begin{center}
\(25\,283\,701,\ 36\,100\,840,\ 42\,531\,528,\ 81\,398\,865,\)
\end{center}

\noindent
have \(3\)-class tower group either\\
\(G\simeq\langle 3^6,54\rangle-\mathbf{\#2;3-\#1;1-\#1;7}\) or\\
\(G\simeq\langle 3^6,54\rangle-\mathbf{\#2;3-\#1;1-\#1;8}\).

\end{enumerate}

\noindent
In each case, the length of the \(3\)-class tower of \(K\) is given by \(\ell_3(K)=3\).

\end{theorem}

\begin{proof}
Since all these real quadratic fields \(K=\mathbb{Q}(\sqrt{d})\) have
\(3\)-capitulation type \(\varkappa_1(K)=(2034)\), \(1^{\mathrm{st}}\) IPAD  
\(\tau^{(1)}(K)=\lbrack 1^2;21,\mathbf{3^2},(21)^2\rbrack\)
and suitable \(2^{\mathrm{nd}}\) IPAD,
the claim is a consequence of Corollary
\ref{cor:c21ES1}.
\end{proof}

\begin{remark}
\label{rmk:RQFc21ES1}
Of course, the percentages given in Theorem
\ref{thm:RQFc21ES1}
are unable to predict reliable tendencies for extensive statistical ensembles.
The results of M.R. Bush show that there are \(146\) real quadratic fields
with discriminants \(0<d<10^9\) having the IPAD in Proposition
\ref{prp:RQFc21ES1}.
For the remaining \(134\) cases outside of the range \(0<d<10^8\),
we cannot specify the \(3\)-tower group,
since the computation would require too much CPU time.
However, according to Proposition
\ref{prp:c21ES1}
and Corollary
\ref{cor:c21ES1},
we know that the length of the \(3\)-tower is exactly \(\ell_3(K)=3\). 
\end{remark}



\subsection{The 2\({}^{\text{nd}}\) Excited State of Capitulation Type c.21}
\label{ss:c21ES2}

As before, \(K\) is an algebraic number field
with \(3\)-class group \(\mathrm{Cl}_3(K)\) of type \((3,3)\hat{=}1^2\).
Let \(L_1,\ldots,L_4\) be the four unramified cyclic cubic extensions of \(K\),
and suppose that the \(1^{\mathrm{st}}\) order Artin pattern \(\mathrm{AP}^{(1)}(K)\) of \(K\)
is given by the IPAD 
\(\tau^{(1)}(K)=\lbrack 1^2;21,\mathbf{4^2},21,21\rbrack\)
and the IPOD
\(\varkappa^{(1)}(K)=\lbrack G^\prime;H_2,\mathbf{G},H_3,H_4\rbrack\),
i.e. \(\tau_1(K)=((9,3),(81,81),(9,3),(9,3))\) are the type invariants of the \(3\)-class groups of the \(L_i\)
and \(\varkappa_1(K)=(2034)\) is the \(3\)-capitulation type of \(K\) in the \(L_i\).

\begin{proposition}
\label{prp:c21ES2}

(D.C. Mayer, \(2010\))

\begin{enumerate}

\item
The second \(3\)-class group \(\mathfrak{G}=\mathrm{G}_3^2(K)\) of \(K\)
is isomorphic to the metabelian \(3\)-group \(\langle 2187,303\rangle(-\#1;1)^3\)
and thus its relation rank satisfies the inequality
\(d_2(\mathfrak{G})\ge d_1(\mathfrak{G})+2\).

\item
Consequently,
if \(K\) is a number field with torsionfree unit rank \(r=1\)
and does not contain the third roots of unity
(in particular, if \(K=\mathbb{Q}(\sqrt{d})\) is a real quadratic field),
then the length of the \(3\)-class field tower of \(K\) must be \(\ell_3(K)\ge 3\).

\end{enumerate}

\end{proposition}

\begin{proof}

Again, a great deal of the proof is similar to the proof of Proposition
\ref{prp:c21GS}.

\begin{enumerate}

\item
First, we prove that \(\mathfrak{G}\) must have coclass \(\mathrm{cc}(\mathfrak{G})=2\).
Next, we show that \(\mathfrak{G}\) must be a descendant of \(\langle 3^5,8\rangle\),
this time using the polarized component \(4^2\) of order \(3^8\) of \(\tau_1(\mathfrak{G})\).
Again, we have
\(d_2(\mathfrak{G})\ge\mathrm{MR}(\mathfrak{G})=4=d_1(\mathfrak{G})+2\),
since the \(p\)-multiplicator rank is a lower bound for the relation rank.
Finally, the polarization \(4^2\) of \(\tau_1(\mathfrak{G})=(21,4^2,(21)^2)\)
determines the nilpotency class \(\mathrm{cl}(\mathfrak{G})=8\)
and thus also the order \(\lvert\mathfrak{G}\rvert=3^{10}\).
The polarization \(0\) of \(\varkappa_1(\mathfrak{G})=(2034)\) with stable components \((234)\)
unambiguously identifies the mainline vertex of order \(3^{10}\),
which is \(\langle 2187,303\rangle(-\#1;1)^3\), according to Figure
\ref{fig:C21PrunedTreeU}.

\item
As before,
\cite[Thm.6, p.140]{Sh}
implies that the relation rank \(d_2(G)\) of the \(3\)-class tower group
\(G:=\mathrm{Gal}(\mathrm{F}_3^\infty(K)\vert K)\)
of a number field \(K\) with unit rank \(r=r_1+r_2-1=1\) and \(\zeta_3\notin K\)
satisfies \(d_2(G)\le d_1(G)+r=2+1=3\),
whence \(G\not\simeq\mathfrak{G}\) and \(\ell_3(K)=\mathrm{dl}(G)\ge 3\).
\end{enumerate}
\end{proof}

\noindent
Suppose now that, under the assumptions preceding Proposition
\ref{prp:c21ES2},
\(G=\mathrm{G}_3^\infty(K)\) denotes the \(3\)-class tower group of \(K\) and
we are additionally given the \(2^{\mathrm{nd}}\) order IPAD of \(K\),
\[\tau^{(2)}(K)=
\lbrack 1^2;((21;\tau_1(L_1)),\mathbf{4^2};\tau_1(L_2)),(21;\tau_1(L_3)),(21;\tau_1(L_4))\rbrack,\]
with fixed \(\tau_1(L_2)=((431)^4)\).

\begin{theorem}
\label{thm:c21ES2}

(D.C. Mayer, Sep. \(2015\))

\begin{enumerate}

\item
\(\tau_1(L_i)=(431,\mathbf{(21)^3)}\) for \(i=1,3,4\)
\(\Longleftrightarrow\)\\
\(G\simeq\langle 3^7,303\rangle\mathbf{(-\#1;1)^3}\), of order \(\lvert G\rvert=3^{10}\),
or\\
\(G\simeq\langle 3^7,303\rangle\mathbf{(-\#1;1)^3-\#1;7}\), of order \(\lvert G\rvert=3^{11}\),

\item
\(\tau_1(L_i)=(431,\mathbf{(31)^3})\) for \(i=1,3,4\)
\(\Longleftrightarrow\)\\
\(G\simeq\langle 3^6,54\rangle\mathbf{-\#2;3(-\#1;1)^3}\), of order \(\lvert G\rvert=3^{11}\),
or\\
\(G\simeq\langle 3^6,54\rangle\mathbf{-\#2;3(-\#1;1)^3-\#1;7}\), of order \(\lvert G\rvert=3^{12}\),
or\\
\(G\simeq\langle 3^6,54\rangle\mathbf{-\#2;3-\#1;1-\#2;1-\#1;1}\), of order \(\lvert G\rvert=3^{12}\),
or\\
\(G\simeq\langle 3^6,54\rangle\mathbf{-\#2;3-\#1;1-\#2;1-\#1;1-\#1;7}\), of order \(\lvert G\rvert=3^{13}\),
or\\
\(G\simeq\langle 3^6,54\rangle\mathbf{-\#2;3-\#1;1-\#2;1-\#1;1-\#1;8}\), of order \(\lvert G\rvert=3^{13}\).

\end{enumerate}

\end{theorem}

\begin{proof}
Similar as in the proof of Theorem
\ref{thm:c21GS},
we construct the descendant tree \(\mathcal{T}(R)\) of the root \(R=\langle 243,8\rangle\),
determine the iterated IPAD of second order \(\tau^{(2)}(V)\) for each vertex \(V\),
and check the second derived quotient \(H/H^{\prime\prime}\) of non-metabelian vertices \(H\).

It turns out that the cover of the common metabelianization
\(\mathfrak{G}=\langle 3^7,303\rangle(-\#1;1)^3=G/G^{\prime\prime}\)
of all candidates \(G\) for the \(3\)-tower group is given by
\(\mathrm{cov}(\langle 3^7,303\rangle(-\#1;1)^3)=\)\\
\(\lbrace\langle 3^7,303\rangle(-\#1;1)^3,\ \langle 3^7,303\rangle(-\#1;1)^3-\#1;7,\)\\
\(\langle 3^6,54\rangle-\#2;1-\#1;1-\#1;2-\#1;1,\ \langle 3^6,54\rangle-\#2;1-\#1;1-\#1;2-\#1;1-\#1;7,\)\\
\(\langle 3^6,54\rangle(-\#2;1-\#1;1)^2,\ \langle 3^6,54\rangle(-\#2;1-\#1;1)^2-\#1;1,\ \langle 3^6,54\rangle(-\#2;1-\#1;1)^2-\#1;8\rbrace\).

The various vertices 
are not characterized uniquely by their iterated IPAD of second order.
Rather they can be identified as batches of two resp. five vertices
in the claimed manner.
\end{proof}

\begin{corollary} 
\label{cor:c21ES2}

(D.C. Mayer, Sep. \(2015\))\\
If \(K\) has torsionfree unit rank \(1\) and does not contain a primitive third root of unity, then

\begin{enumerate}

\item
\(\tau_1(L_i)=(431,\mathbf{(21)^3)}\) for \(i=1,3,4\)
\(\Longleftrightarrow\)\\
\(G\simeq\langle 3^7,303\rangle\mathbf{(-\#1;1)^3-\#1;7}\).

\item
\(\tau_1(L_i)=(431,\mathbf{(31)^3})\) for \(i=1,3,4\)
\(\Longleftrightarrow\)\\
\(G\simeq\langle 3^6,54\rangle\mathbf{-\#2;3(-\#1;1)^3-\#1;7}\)
or\\
\(G\simeq\langle 3^6,54\rangle\mathbf{-\#2;3-\#1;1-\#2;1-\#1;1-\#1;7}\)
or\\
\(G\simeq\langle 3^6,54\rangle\mathbf{-\#2;3-\#1;1-\#2;1-\#1;1-\#1;8}\).

\end{enumerate}

\noindent
All these groups have derived length \(\mathrm{dl}(G)=3\).

\end{corollary}

\begin{proof}
The three groups
\(\langle 3^6,54\rangle\mathbf{-\#2;3-\#1;1-\#2;1-\#1;1}\),
\(\langle 3^6,54\rangle-\mathbf{\#2;1(-\#1;1)^3}\)
and
\(\langle 3^7,303\rangle\mathbf{(-\#1;1)^3}\),
which are infinitely capable,
have \(p\)-multiplicator rank \(\mathrm{MR}(G)=4\) and
thus relation rank \(d_2(G)\ge 4\),
and consequently cannot satisfy the Shafarevich inequality \(d_2(G)\le d_1(G)+r=2+1=3\)
\cite[Thm.6, p.140]{Sh}
for a field \(K\) with unit rank \(r=r_1+r_2-1=1\) and \(\zeta_3\notin K\).
\end{proof}



\subsection{Real Quadratic Fields of Type c.21\(\uparrow^2\)}
\label{ss:RQFc21ES2}

\begin{proposition} 
\label{prp:RQFc21ES2}

(M.R. Bush, Jul. \(2015\))\\
In the range \(0<d<10^9\) of fundamental discriminants \(d\)
of real quadratic fields \(K=\mathbb{Q}(\sqrt{d})\)
there exist precisely \(\mathbf{5}\) cases with \(1^{\mathrm{st}}\) IPAD
\(\tau^{(1)}(K)=\lbrack 1^2;21,\mathbf{4^2},(21)^2\rbrack\).

\end{proposition}

\begin{proof}
The results were communicated to us on July 11, 2015, by M.R. Bush, who used PARI/GP
\cite{PARI}
with similar techniques as described in our paper
\cite[\S\ 5, pp.446--450]{Ma3},
and additionally double-checked with MAGMA
\cite{MAGMA}.
\end{proof}

\begin{corollary} 
\label{cor:RQFc21ES2}

(D.C. Mayer, \(2010\))\\
A quadratic field \(K=\mathbb{Q}(\sqrt{d})\)
with
\(\tau^{(1)}(K)=\lbrack 1^2;21,\mathbf{4^2},(21)^2\rbrack\)
must be a real quadratic field
with \(3\)-capitulation type \(\varkappa_1(K)=(2034)\).

\end{corollary}

\begin{proof}
This is a consequence of item (1) in
\cite[Cor.4.4.3, p.442]{Ma3}.
\end{proof}

\bigskip
\noindent
\begin{theorem} 
\label{thm:RQFc21ES2}

(D.C. Mayer, Sep. \(2015\))

\begin{enumerate}

\item
The \(\mathbf{4}\) real quadratic fields \(K=\mathbb{Q}(\sqrt{d})\) with the following discriminants \(d\) (\(\mathbf{80}\%\) of \(5\)),

\begin{center}
\(407\,086\,012,\ 509\,164\,587,\ 510\,908\,876,\ 870\,946\,856\),   
\end{center}

\noindent
have \(3\)-class tower group
\(G\simeq\langle 3^7,303\rangle\mathbf{(-\#1;1)^3-\#1;7}\).

\item
The \textbf{single} real quadratic fields \(K=\mathbb{Q}(\sqrt{d})\) with discriminant \(d=\) (\(\mathbf{20}\%\) of \(5\)),
has \(3\)-class tower group either\\
\(G\simeq\langle 3^6,54\rangle\mathbf{-\#2;3(-\#1;1)^3-\#1;7}\) or\\
\(G\simeq\langle 3^6,54\rangle\mathbf{-\#2;3-\#1;1-\#2;1-\#1;1-\#1;7}\) or\\
\(G\simeq\langle 3^6,54\rangle\mathbf{-\#2;3-\#1;1-\#2;1-\#1;1-\#1;8}\).

\end{enumerate}

\noindent
In each case, the length of the \(3\)-class tower of \(K\) is given by \(\ell_3(K)=3\).

\end{theorem}

\begin{proof}
Since all these real quadratic fields \(K=\mathbb{Q}(\sqrt{d})\) have
\(3\)-capitulation type \(\varkappa_1(K)=(2034)\), \(1^{\mathrm{st}}\) IPAD  
\(\tau^{(1)}(K)=\lbrack 1^2;21,\mathbf{4^2},(21)^2\rbrack\)
and suitable \(2^{\mathrm{nd}}\) IPAD,
the claim is a consequence of Corollary
\ref{cor:c21ES2}.
\end{proof}



\section{Proof of the main result}
\label{s:ProofMainResult}

\noindent
Combining several remarks and the results of two theorems,
we are finally able to prove our initial Main Theorem
\ref{thm:Main}.

\begin{proof}
According to Remark
\ref{rmk:RQFc18GS},
Remark
\ref{rmk:RQFc18ES1}
and Theorem
\ref{thm:RQFc18ES2},
there exist exactly \(4318+138+5=4461\) real quadratic fields \(K=\mathbb{Q}(\sqrt{d})\)
of type c.18, \(\varkappa_1(K)=(0122)\),
and according to Remark
\ref{rmk:RQFc21GS},
Remark
\ref{rmk:RQFc21ES1}
and Theorem
\ref{thm:RQFc21ES2},
there exist exactly \(4377+146+5=4528\) real quadratic fields \(K=\mathbb{Q}(\sqrt{d})\)
of type c.21, \(\varkappa_1(K)=(2034)\),
in the range of fundamental discriminants \(0<d<10^9\).
They all have a \(3\)-class field tower of exact length \(\ell_3(K)=3\).
\end{proof}



\section{When do we consider the \(p\)-class tower as \lq\lq known\rq\rq ?}
\label{s:KnownClsFldTow}

\noindent
As we have seen in this article,
there are various pieces of information contributing to the knowledge
of the \(p\)-class tower \(\mathrm{F}_p^\infty(k)\) of a number field \(k\):

\begin{enumerate}
\item
the \(p\)-class tower group \(G=\mathrm{Gal}(\mathrm{F}_p^\infty(k)\vert k)\) of \(k\),
\item
the length \(\ell_p(k)=\mathrm{dl}(G)\) of the \(p\)-class tower of \(k\),
\item
iterated IPADs \(\tau^{(n)}(G)\) and IPODs \(\varkappa^{(n)}(G)\) of higher order \(n\ge 2\),
\item
the second \(p\)-class group
\(\mathfrak{G}=\mathrm{Gal}(\mathrm{F}_p^2(k)\vert k)\simeq G/G^{\prime\prime}\) of \(k\),
\item
the order \(\lvert\mathfrak{G}\rvert\), nilpotency class \(\mathrm{cl}(\mathfrak{G})\)
and coclass  \(\mathrm{cc}(\mathfrak{G})\) of \(\mathfrak{G}\),
\item
the annihilator ideal \(\mathfrak{A}\) of the commutator subgroup \(\mathfrak{G}^\prime\)
of \(\mathfrak{G}\),
\item
the \(p\)-capitulation kernels \(\ker(j_{L\vert k})\)
of the unramified cyclic extensions \(L\vert k\)
of degree \(p\), forming the first layer \(\varkappa_1(k)\) of the
\(p\)-capitulation type, or TKT, of \(k\),
\item
the \(p\)-class groups \(\mathrm{Cl}_p(L)\) of the unramified cyclic extensions \(L\vert k\)
of degree \(p\), forming the first layer \(\tau_1(k)\) of the TTT of \(k\),
\item
the \(p\)-class group \(\mathrm{Cl}_p(k)\simeq G/G^\prime\) of \(k\),
\item
the \(p\)-class rank \(r_p(k)=\dim_{\mathbb{F}_p}(\mathbb{F}_p\otimes_{\mathbb{Z}_p}\mathrm{Cl}_p(k))\) of \(k\).
\end{enumerate}

\begin{figure}[ht]
\caption{Pieces of information on a \(p\)-class tower \(\mathrm{F}_p^\infty(k)\)}
\label{fig:ClassTowerInfo}


\setlength{\unitlength}{1cm}
\begin{picture}(10,10)(-5,1)

\put(0,10){\makebox(0,0)[cc]{(1) \(G=\mathrm{G}_p^\infty(k)\)}}

\put(0,9.5){\vector(-1,-1){1}}
\put(-1,8){\makebox(0,0)[rc]{(2) \(\ell_p(k)=\mathrm{dl}(G)\)}}
\put(0,9.5){\vector(1,-1){1}}
\put(1,8){\makebox(0,0)[lc]{(3) \(\tau^{(n)}(G),\varkappa^{(n)}(G)\)}}

\put(0,9.5){\vector(0,-1){2}}
\put(0,7){\makebox(0,0)[cc]{(4) \(\mathfrak{G}=\mathrm{G}_p^2(k)\simeq G/G^{\prime\prime}\)}}

\put(0,6.5){\vector(-3,-1){3}}
\put(-3,5){\makebox(0,0)[rc]{(5) \(\mathrm{cl}(\mathfrak{G}^\prime),\mathrm{cc}(\mathfrak{G}^\prime)\)}}
\put(0,6.5){\vector(3,-1){3}}
\put(3,5){\makebox(0,0)[lc]{(6) \(\mathfrak{A}(\mathfrak{G}^\prime)\)}}

\put(0,6.5){\vector(-1,-1){2}}
\put(-1,4){\makebox(0,0)[rc]{(7) \(\varkappa_1(k)=\varkappa_1(\mathfrak{G})\)}}
\put(0,6.5){\vector(1,-1){2}}
\put(1,4){\makebox(0,0)[lc]{(8) \(\tau_1(k)=\tau_1(\mathfrak{G})\)}}

\put(0,6.5){\vector(0,-1){3}}
\put(0,3){\makebox(0,0)[cc]{(9) \(\mathrm{Cl}_p(k)\simeq\mathfrak{G}/\mathfrak{G}^\prime\)}}

\put(-3,3.5){\vector(1,-1){2}}
\put(3,3.5){\vector(-1,-1){2}}
\put(0,2.5){\vector(0,-1){1}}
\put(0,1){\makebox(0,0)[cc]{(10) \(r_p(k)\)}}

\end{picture}

\end{figure}
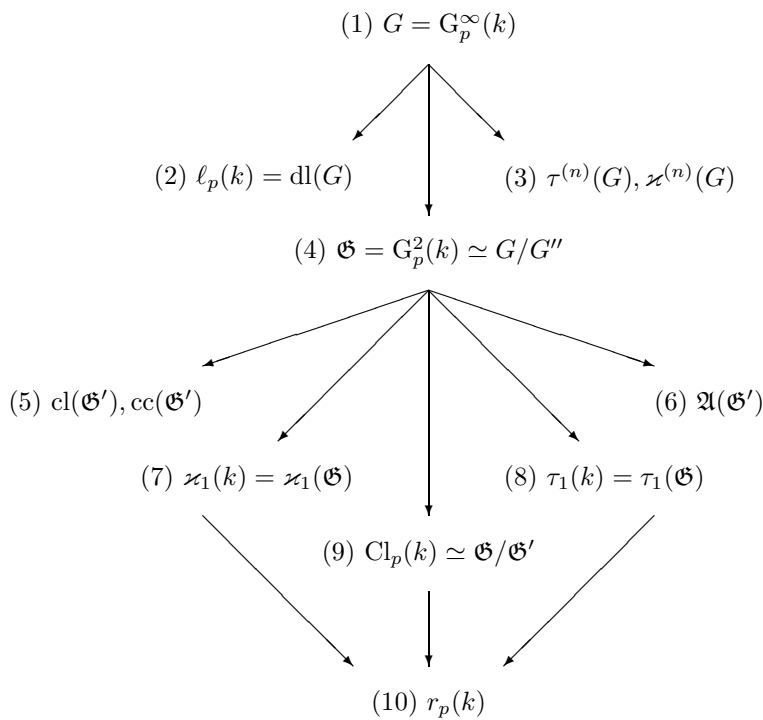



\bigskip
\noindent
Figure
\ref{fig:ClassTowerInfo}
shows the logical implications between the pieces of information,
which are valid generally, since they arise from mappings,
such as \(\mathrm{dl}()\), \(\mathrm{cl}()\), \(\mathrm{cc}()\), and so on.

\bigskip
\noindent
However, in special cases, there exist further logical relations.

\begin{itemize}
\item
If the \(p\)-class rank \(r_p(k)\) of \(k\) is sufficiently large,
then the Golod--Shafarevich Theorem
\cite{GoSh}
of 1964 ensures
an infinite \(p\)-class tower with length \(\ell_p(k)=\infty\),
i.e. (1)\(\rightarrow\)(9).
\item
Conversely, a finite \(p\)-class tower with length \(\ell_p(k)<\infty\)
enforces an upper bound on the \(p\)-class rank \(r_p(k)\) of \(k\)
\cite{GoSh},
i.e. (9)\(\rightarrow\)(1).
\item
The Scholz--Taussky Theorem
\cite{SoTa}
of 1934 ensures
a finite \(p\)-class tower with exact length \(\ell_p(k)=2\),
if the annihilator \(\mathfrak{A}\) is the special ideal \(\mathfrak{L}_2\),
i.e. (5)\(\rightarrow\)(9).
\item
For a group \(\mathfrak{G}\) without defect of commutativity (\(k(\mathfrak{G})=0\)),
the knowledge of the nilpotency class \(\mathrm{cl}(\mathfrak{G})=\alpha+1\)
and the coclass  \(\mathrm{cc}(\mathfrak{G})=\beta\) of \(\mathfrak{G}\)
is equivalent with the knowledge of the annihilator ideal
\(\mathfrak{A}=\mathfrak{R}_{\alpha,\beta}\)
and with the knowledge of the first layer \(\tau_1(K)\) of the TTT of \(\mathfrak{G}\),
i.e. (6)\(\leftrightarrow\)(5) and (6)\(\leftrightarrow\)(3).
\item
For \(p=3\) and \(\mathrm{Cl}_p(k)\simeq (3,3)\),
the TTT \(\tau_1(k)\) occasionally determines the TKT \(\varkappa_1(k)\)
\cite[Cor.4.2.2, p.436, and Cor.4.3.2, p.439]{Ma3},
i.e. (3)\(\rightarrow\)(4).
\item
For \(p=5\) and \(\mathrm{Cl}_p(k)\simeq (5,5)\), 
the TKT \(\varkappa_1(k)\) occasionally determines the TTT \(\tau_1(k)\)
\cite[Tbl.3.11, p.447, and the remark concerning the counter \(\eta\)]{Ma4},
i.e. (4)\(\rightarrow\)(3),
or vice versa
\cite[Thm.3.8 and Tbl.3.3]{Ma4}.
\end{itemize}

\noindent
We are convinced that the \(p\)-class tower \(\mathrm{F}_p^\infty(k)\) of a number field \(k\)
can only be considered as \lq\lq known\rq\rq,
if we are able to give a presentation of the \(p\)-class tower group
\(G=\mathrm{Gal}(\mathrm{F}_p^\infty(k)\vert k)\)
and not merely to specify the \(p\)-tower length \(\ell_p(k)\).
This information permits to compute all data in Figure
\ref{fig:ClassTowerInfo}.
Therefore, we made the initial Main Theorem
\ref{thm:Main}
precise by stating and proving the detailed Theorems
\ref{thm:RQFc18GS},
\ref{thm:RQFc18ES1},
\ref{thm:RQFc21GS} and
\ref{thm:RQFc21ES1}.

In view of future research, we point out that, although it is known
\cite[\S\ 7]{Ma6}
that the complex quadratic field \(k=\mathbb{Q}(\sqrt{d})\) with discriminant \(d=4\,447\,704\)
and \(3\)-class group of type \((3,3,3)\) has a \(3\)-class tower of length \(\ell_3(k)=\infty\),
according to the results of Koch and Venkov
\cite{KoVe},
there will remain a lot of work to do for establishing
a pro-\(3\) presentation of the \(3\)-tower group
\(G=\mathrm{Gal}(\mathrm{F}_3^\infty(k)\vert k)\)
of this mysterious field and determining
the path from the root of the tree \(\mathcal{T}(\langle 8,5\rangle)\)
to its second \(3\)-class group
\(\mathfrak{G}=\mathrm{Gal}(\mathrm{F}_3^2(k)\vert k)\simeq G/G^{\prime\prime}\).



\section{Acknowledgements}
\label{s:Acknowledgements}

\noindent
We gratefully acknowledge that our research is supported financially by the
Austrian Science Fund (FWF): P 26008-N25.

We express our gratitude to M.F. Newman (ANU, Canberra, ACT) for pointing out that
the capable groups \(G\) in Theorems
\ref{thm:c18ES1}
and 
\ref{thm:c21ES1}
have \(p\)-multiplicator rank \(\mu(G)=4\)
and thus satisfy the inequality \(d_2(G)\ge\mu(G)=d_1(G)+2\)
between their relation rank \(d_2(G)\) and generator rank \(d_1(G)=2\).
This disqualifies them as candidates for \(p\)-class tower groups
of real quadratic fields, according to the corrected Shafarevich Theorem
\ref{thm:Shafarevich}.

Sincere thanks are given to Michael R. Bush (WLU, Lexington, VA)
for making available brand-new
numerical results on IPADs of real quadratic fields \(K=\mathbb{Q}(\sqrt{d})\),
and the distribution of discriminants \(d<10^9\) over these IPADs,
computed in July 2015.

We are indebted to Nigel Boston, Michael R. Bush and Farshid Hajir
for kindly making available an unpublished database containing
numerical results of their paper
\cite{BBH}
and a related paper on real quadratic fields, which is still in preparation.

A succinct version of this article
has been presented at the 22nd Czech and Slovak International Conference on Number Theory
in Liptovsk\'y J\'an, Slovakia, on August 31, 2015.




\end{document}